\documentclass[11pt,twoside]{article}

\newcommand{\HeadTitle}{Algebraic representatives of the ratios 
  \ensuremath{\zeta(2n+1)/\pi^{2n}} and 
  \ensuremath{\beta(2n)/\pi^{2n-1}}}

\newcommand{\HeadTitleTwo}{\begin{center}
\Large{\textit{Algebraic representatives of the ratios}}\end{center} 
\begin{center}
\Large{\textit{$\dfrac{\zeta(2n+1)}{\pi^{2n}}$ and $\dfrac{\beta(2n)}{\pi^{2n-1}}$}}
\end{center}}

\usepackage[margin=0.8in]{geometry}
\usepackage{amsmath, amssymb, amsfonts, mathtools}
\usepackage{amsthm}
\usepackage{mathrsfs}
\usepackage{stmaryrd}
\usepackage{esint}
\usepackage{eqnalign}

\usepackage{graphicx}
\usepackage{enumitem}
\usepackage[most]{tcolorbox}
\usepackage{float}
\usepackage{listings}
\usepackage{array}
\usepackage{booktabs}
\usepackage{xcolor}
\usepackage{emptypage}
\usepackage{tensor}
\usepackage{tabularx}

\usepackage{fancyhdr}
\usepackage{hyperref}
\usepackage[nameinlink,noabbrev]{cleveref}

\usepackage{titlesec}
\usepackage[T1]{fontenc}
\usepackage[utf8]{inputenc}
\usepackage{lmodern}
\usepackage{titling}

\usepackage[backend=biber,style=numeric]{biblatex}
\newcommand{\EulerB}[2]{\left\langle {#1 \atop #2} \right\rangle^{\!B}}
\newcommand{\Arg}{\operatorname{Arg}}
\providecommand{\HeadTitle}{} 
\providecommand{\HeadTitleTwo}{} 
\providecommand{\HeadAuthor}{Luc Ramsès TALLA WAFFO} 

\titleformat{\section}[block]
  {\normalfont\large\bfseries\itshape\centering}
  {§\thesection.}
  {1em}
  {}

\titleformat{\subsection}
  {\normalfont\normalsize\bfseries}
  {\thesubsection}
  {1em}
  {}

\newtheorem{theorem}{Theorem}[section]
\newtheorem{lemma}[theorem]{Lemma}
\newtheorem{proposition}[theorem]{Proposition}
\newtheorem{corollary}[theorem]{Corollary}

\newtheorem{remark}[theorem]{Remark}

\crefname{example}{example}{examples}
\Crefname{example}{Example}{Examples}

\crefname{corollary}{corollary}{corollaries}
\Crefname{corollary}{Corollary}{Corollaries}

\crefname{definition}{definition}{definitions}
\Crefname{definition}{Definition}{Definitions}

\crefname{remark}{remark}{remarks}
\Crefname{remark}{Remark}{Remarks}

\crefname{conjecture}{conjecture}{conjectures}
\Crefname{conjecture}{Conjecture}{Conjectures}

\crefname{lemma}{lemma}{lemmas}
\Crefname{lemma}{Lemma}{Lemmas}

\crefname{proposition}{proposition}{propositions}
\Crefname{proposition}{Proposition}{Propositions}

\crefname{theorem}{theorem}{theorems}
\Crefname{theorem}{Theorem}{Theorems}

\numberwithin{equation}{section}

\usepackage{fontspec}
\begin{document}

\thispagestyle{fancy}

\vspace{0.2cm}

\begin{center}
\Large{\HeadTitleTwo}
\end{center}

\hspace{3cm}

\begin{center}
Luc Ramsès TALLA WAFFO \\
Technische Universität Darmstadt\\
Karolinenplatz 5, 64289 Darmstadt, Germany\\
ramses.talla@stud.tu-darmstadt.de\\
\vspace{0.5cm}
February 18, 2026
\end{center}

\begin{abstract}
In \cite{TallaWaffo2025arxiv2511.02843} we introduced even polynomials $\Xi_n, \Lambda_n\in\mathbb{Q}[x]$ arising from integral representations of
$\beta(2n)/\pi^{2n-1}$ and $\zeta(2n+1)/\pi^{2n}$. In this paper we give explicit
closed formulae for these polynomials in terms of Eulerian numbers and study
their structural properties. These properties may prove useful in studying the arithmetic nature of the ratios $\beta(2n)/\pi^{2n}$ and $\zeta(2n+1)/\pi^{2n+1}.$
\end{abstract}

\vspace{0.2cm}

\paragraph{Notations.}
Throughout this manuscript, we adopt the following conventions. The set $\mathbb{N}_0$ denotes the set of all non-negative integers, while $\mathbb{N}$ denotes the set of strictly positive integers. The symbol $\displaystyle \genfrac{\langle}{\rangle}{0pt}{}{m}{k}$ denotes the Eulerian numbers of type $A$. The quantities $B_{2n}$ and $E_{2n}$ denote, respectively, the Bernoulli numbers and the Euler (secant) numbers. Finally, $\displaystyle \EulerB{n}{k}$ denotes the Eulerian numbers of type $B$.

\vspace{0.5cm}

\section*{Introduction}
\addcontentsline{toc}{section}{Introduction}

Euler’s classical evaluation
\(
\zeta(2)=\dfrac{\pi^{2}}{6}
\)
and, more generally, his formula
\[
\zeta(2n)=(-1)^{n+1}\frac{B_{2n}(2\pi)^{2n}}{2(2n)!}
\]
for the even zeta values stand among the landmarks of analytic number theory. By contrast, no comparable closed formula is known for the odd values \(\zeta(2n+1)\). A similar dichotomy appears for Dirichlet’s beta function
\[
\beta(s)=\sum_{m=0}^{\infty}\frac{(-1)^m}{(2m+1)^s},
\]
whose odd special values admit explicit evaluations in terms of Euler numbers, whereas the even values \(\beta(2n)\) remain far less understood. The arithmetic nature of the normalized quantities
\[
\frac{\zeta(2n+1)}{\pi^{2n+1}},
\qquad
\frac{\beta(2n)}{\pi^{2n}}
\]
thus continues to raise deep and largely open questions.

\vspace{0.2cm}

In the earlier work \cite{TallaWaffo2025arxiv2511.02843}, integral representations for the normalized values \(\zeta(2n+1)/\pi^{2n}\) and \(\beta(2n)/\pi^{2n-1}\) were derived in a unified analytic framework. In particular, it was shown that for every \(n\in\mathbb{N}\), there exist even polynomials \(\Xi_n,\Lambda_n\in\mathbb{Q}[x]\) such that
\[
\frac{\beta(2n)}{\pi^{2n-1}}
=
\int_0^1 \frac{x\,\Xi_n(x)}{\sqrt{1-x^2}\,\operatorname{arctanh}(x)}\,dx,
\qquad
\frac{\zeta(2n+1)}{\pi^{2n}}
=
\int_0^1 \frac{x\,\Lambda_n(x)}{\operatorname{arctanh}(x)}\,dx.
\]
These polynomial families arise naturally from Malmsten-type integrals and from rational representations of polylogarithms of negative order in terms of Eulerian numbers of type \(A\) and type \(B\).

\vspace{0.2cm}

The aim of the present paper is to study these polynomials systematically. We derive explicit closed formulae for \(\Xi_n\) and \(\Lambda_n\), establish their main algebraic properties, and analyze their zeros in detail. In particular, we determine their leading coefficients and boundary values, prove that all real zeros lie in \([-1,1]\), and show that for \(n\ge2\) all zeros are real and simple. We also prove interlacing properties between consecutive polynomials and derive consequences such as log-concavity and unimodality of the coefficient sequences.

\vspace{0.2cm}

A substantial part of the paper is devoted to the asymptotic zero distribution. In \Cref{part:asymptotic_zero_distribution} we show that the real zeros of the normalized polynomials associated with \(\Xi_n\) and \(\Lambda_n\) become equidistributed with respect to the same explicit limiting law on \((0,1)\). We determine both the limiting density and the corresponding distribution function, and from these obtain quantitative information on the behavior of the smallest and largest zeros. This reveals a strong accumulation of zeros near the endpoints, with different asymptotic regimes at \(0\) and \(1\).

\vspace{0.2cm}

In this way, the paper extends the analytic framework of \cite{TallaWaffo2025arxiv2511.02843} and brings out a rich interaction between special values of Dirichlet series, Eulerian combinatorics, real-rooted polynomials, and asymptotic zero distributions.

\vspace{0.2cm}

The paper is organized as follows:
\begin{itemize}
  \item In \Cref{part:mathematical_motives} we explain how these polynomial families arise in connection with the irrationality questions surrounding the ratios \(\zeta(2n+1)/\pi^{2n+1}\) and \(\beta(2n)/\pi^{2n}\).
  \item In \Cref{part:preliminaries} we collect auxiliary lemmas used throughout the paper.
  \item In \Cref{part:existence_and_closed_forms} we recover the polynomials and derive explicit closed-form expressions for their coefficients.
  \item In \Cref{part:algebraic_and_analytic_structure} we establish their main structural properties, including real-rootedness and interlacing.
  \item In \Cref{part:asymptotic_zero_distribution} we study the limiting distribution of their zeros and derive the explicit asymptotic laws governing them.
\end{itemize}

\part{Mathematical motivation}\label[part]{part:mathematical_motives}

A key ingredient underlying our work is the very simple identity
\[
\int_0^1 \frac{x^{2m}}{\sqrt{1-x^2}}\,dx
= \frac{1}{2}\,B\!\left(m+\frac{1}{2},\frac{1}{2}\right)
= \frac{\sqrt{\pi}}{2}\,\frac{\Gamma\!\left(m+\tfrac{1}{2}\right)}{\Gamma(m+1)}
= \frac{\pi}{2}\,\frac{(2m-1)!!}{(2m)!!},
\]
valid for every $m\in\mathbb{N}_0$. Here $n!!$ denotes the double factorial, i.e.
\[
(2m-1)!! = 1\cdot 3\cdot 5\cdots (2m-1),\qquad
(2m)!! = 2\cdot 4\cdot 6\cdots (2m),
\]
with the usual conventions $(-1)!! = 0!! = 1$. The identity itself follows by the substitution $t=x^2$, which transforms the integral into
\[
\frac{1}{2}\int_0^1 t^{m-\frac12}(1-t)^{-\frac12}\,dt
= \frac{1}{2}\,B\!\left(m+\frac12,\frac12\right),
\]
and then by using the relation between the Beta and Gamma functions together with the standard formulae expressing $\Gamma(m+\tfrac12)/\Gamma(m+1)$ in terms of double factorials. As a direct consequence, the quantity $\displaystyle\int_0^1 \frac{P_n(x)}{\sqrt{1-x^2}}\,dx$ is always a rational multiple of $\pi$ when $P_n$ is an even polynomial, that is, a polynomial of the form:
\begin{equation}\label{eq:form}a_n\,x^{2n} + a_{n-1}\,x^{2n-2} + a_{n-2}\,x^{2n-4} + \ldots + a_1\,x^{2} + a_0 \quad \text{where} \quad \begin{pmatrix}
a_0\\
a_1\\
\vdots\\
a_n
\end{pmatrix}
 \in \mathbb{Q}^{n+1}.\end{equation} In \cite{TallaWaffo2025arxiv2511.02843} we proved the existence of polynomials $\Xi_n(x)$ and $\Lambda_n(x)$ of the kind \eqref{eq:form} such that the following representations hold:
\begin{equation}\label{eq:integral_definition}
\frac{\beta(2n)}{\pi^{2n-1}}=\int_0^1 \frac{x\,\Xi_n(x)}{\sqrt{1-x^2}\,\operatorname{arctanh}x}\,dx,
\qquad
\frac{\zeta(2n+1)}{\pi^{2n}}=\int_0^1 \frac{x\,\Lambda_n(x)}{\operatorname{arctanh}x}\,dx.
\end{equation}  We deduce that the integrals
\(\displaystyle
\int_0^1 \frac{\Lambda_n(x)}{\sqrt{1-x^2}}\,dx
\qquad\text{and}\qquad
\int_0^1 \frac{\Xi_n(x)}{\sqrt{1-x^2}}\,dx
\)
are always rational multiples of $\pi$ for every $n \in \mathbb{N}$. By the table of \cref{prop:explicit_formulae_selected_weights}, we even deduce that these quantities are strictly positive. That is, there exist $(p_n^{A}, q_n^{A}), (p_n^{B}, q_n^{B}) \in \mathbb{N}^{2}$ satisfying
\begin{equation}\label{eq:integral_pi_one}
\int_0^1 \frac{\Lambda_n(x)}{\sqrt{1-x^2}}\,dx = \frac{p_n^{A}}{q_n^{A}} \pi
\qquad\text{and}\qquad
\int_0^1 \frac{\Xi_n(x)}{\sqrt{1-x^2}}\,dx=\frac{p_n^{B}}{q_n^{B}}\pi.
\end{equation}
The uniform convergence of these polynomials to $0$ on $(0,1)$ established in \cref{theorem:uniform-to-zero} implies that 
\[\frac{p_n^{A}}{q_n^{A}} \to 0 
\qquad\text{and}\qquad
\frac{p_n^{B}}{q_n^{B}} \to 0 \quad \text{as } n\to\infty. 
\] Now assuming that $\dfrac{\beta(2n)}{\pi^{2n}}$ and $\dfrac{\zeta(2n+1)}{\pi^{2n+1}}$ are rational, we may write
\[
\dfrac{\beta(2n)}{\pi^{2n-1}} = \frac{s_n^{B}}{t_n^{B}} \pi
\qquad\text{and}\qquad
\dfrac{\zeta(2n+1)}{\pi^{2n}} = \frac{s_n^{A}}{t_n^{A}} \pi.
\]
with some positive integers $s_n^{A}, s_n^{B}, t_n^{A}$ and $t_n^{B}$. Using \eqref{eq:integral_pi_one} yields
\[
\dfrac{\beta(2n)}{\pi^{2n-1}} = \frac{s_n^{B}\,q_n^{B}}{t_n^{B}\,p_n^{B}} \int_0^1 \frac{\Xi_n(x)}{\sqrt{1-x^2}}\,dx
\qquad\text{and}\qquad
\dfrac{\zeta(2n+1)}{\pi^{2n}} = \frac{s_n^{A} \,q_n^{A}}{t_n^{A}\,p_n^{A}} \int_0^1 \frac{\Lambda_n(x)}{\sqrt{1-x^2}}\,dx.
\]

Going back to \eqref{eq:integral_definition} gives

\[
\int_0^1 \frac{x\,\Xi_n(x)}{\sqrt{1-x^2}\,\operatorname{arctanh}x}\,dx = \frac{s_n^{B}\,q_n^{B}}{t_n^{B}\,p_n^{B}} \int_0^1 \frac{\Xi_n(x)}{\sqrt{1-x^2}}\,dx
\quad\text{,}\quad
\int_0^1 \frac{x\,\Lambda_n(x)}{\operatorname{arctanh}x}\,dx = \frac{s_n^{A} \,q_n^{A}}{t_n^{A}\,p_n^{A}} \int_0^1 \frac{\Lambda_n(x)}{\sqrt{1-x^2}}\,dx.
\]
This is equivalent to
\[
\begin{cases}
\displaystyle\int_0^1 \Xi_n(x)\left(\frac{x\,}{\sqrt{1-x^2}\,\operatorname{arctanh}x} - \frac{s_n^{B}\,q_n^{B}}{t_n^{B}\,p_n^{B}} \frac{1}{\sqrt{1-x^2}}\right)\,dx = 0\\
\vspace{0.3cm}\\
\displaystyle\int_0^1 \Lambda_n(x)\left(\frac{x\,}{\operatorname{arctanh}x}\,-\frac{s_n^{A} \,q_n^{A}}{t_n^{A}\,p_n^{A}} \frac{1}{\sqrt{1-x^2}}\right)\,dx = 0.
\end{cases}
\]

These last equalities may serve as a natural starting point for a more
robust approach to proving the irrationality of the ratios
\(\beta(2n)/\pi^{2n}\) and \(\zeta(2n+1)/\pi^{2n+1}\). In particular, a
detailed analysis of the structural and analytic properties of the
polynomials involved is likely to yield further insight into this
problem. The present article is devoted to the investigation of the
structure of these polynomials and of several closely related properties.

\part{Preliminaries}\label[part]{part:preliminaries}

\begin{lemma}\label[lemma]{lemma:factorial_root_goes_to_infty}
As $n\to\infty$ one has
\[
\bigl((2n-1)!\bigr)^{-\frac{1}{n-1}}\longrightarrow 0.
\qquad\text{,}\qquad
\frac{4}{\bigl(2(2n)!\bigr)^{\frac{1}{n-1}}}\longrightarrow 0.
\]
\end{lemma}

\begin{proof}
Choose $N\in\mathbb{N}$ such that for all $m\ge N$ one has
$\displaystyle m!\ge \left(\frac{m}{e}\right)^m$.
Applying this with $m=2n-1$ gives, for $n$ large enough,
$\displaystyle (2n-1)!\ge \left(\frac{2n-1}{e}\right)^{2n-1}$, hence
$\displaystyle \bigl((2n-1)!\bigr)^{\frac{1}{n-1}}
\ge \left(\frac{2n-1}{e}\right)^{\frac{2n-1}{n-1}}
= \left(\frac{2n-1}{e}\right)^{2+\frac{1}{n-1}}\to\infty$.
Therefore $\displaystyle \bigl((2n-1)!\bigr)^{-\frac{1}{n-1}}\to 0$.

Similarly,
$\displaystyle \bigl(2(2n)!\bigr)^{\frac{1}{n-1}}
=2^{\frac{1}{n-1}}\bigl((2n)!\bigr)^{\frac{1}{n-1}}$
with $\displaystyle 2^{\frac{1}{n-1}}\to 1$, and applying the same bound with $m=2n$ yields
$\displaystyle (2n)!\ge \left(\frac{2n}{e}\right)^{2n}$, so
$\displaystyle \bigl((2n)!\bigr)^{\frac{1}{n-1}}
\ge \left(\frac{2n}{e}\right)^{\frac{2n}{n-1}}
= \left(\frac{2n}{e}\right)^{2+\frac{2}{n-1}}\to\infty$.
Hence $\displaystyle \bigl(2(2n)!\bigr)^{\frac{1}{n-1}}\to\infty$, and thus
$\displaystyle \frac{4}{\bigl(2(2n)!\bigr)^{\frac{1}{n-1}}}\to 0$.
\end{proof}

\begin{lemma}\label[lemma]{lemma:Enk}
Let $k\in\mathbb{N}_0$, $n\in\mathbb{N}$ with $0\le k\le n-1$. Define
\[
E_{n,k}(x):=(1-x^2)^k\Big((1+x)^{2n-2k-1}-(1-x)^{2n-2k-1}\Big).
\]
Then $E_{n,k}$ is an odd polynomial of degree $2n-1$ and admits the representation
\[
E_{n,k}(x)
=
2x\sum_{t=0}^{n-1}
\left(
\sum_{i=0}^{k}
(-1)^i
\binom{k}{i}
\binom{2n-2k-1}{\,2t-2i+1\,}
\right)
x^{2t},
\]
where, throughout, binomial coefficients are understood to vanish whenever the lower index is negative or exceeds the upper index.
\end{lemma}

\begin{proof}
Set $\displaystyle P_k(x):=(1-x^2)^k$ and $\displaystyle Q_{n,k}(x):=(1+x)^{2n-2k-1}-(1-x)^{2n-2k-1}$. By the binomial theorem,
$\displaystyle P_k(x)=\sum_{i=0}^{k}\binom{k}{i}(-1)^i x^{2i}$ and
$\displaystyle Q_{n,k}(x)=2\sum_{r=0}^{n-k-1}\binom{2n-2k-1}{2r+1}x^{2r+1}$.
With the convention $\displaystyle \binom{a}{b}=0$ for $b<0$ or $b>a$, this becomes
$\displaystyle P_k(x)=\sum_{i=0}^{\infty}\binom{k}{i}(-1)^i x^{2i}$ and
$\displaystyle Q_{n,k}(x)=2\sum_{r=0}^{\infty}\binom{2n-2k-1}{2r+1}x^{2r+1}$.
Thus
$\displaystyle P_k(\sqrt{x})=\sum_{i=0}^{\infty}\binom{k}{i}(-1)^i x^i$ and
$\displaystyle Q_{n,k}(\sqrt{x})=2\sqrt{x}\sum_{r=0}^{\infty}\binom{2n-2k-1}{2r+1}x^r$, so by the Cauchy product
\[
P_k(\sqrt{x})Q_{n,k}(\sqrt{x})
=
2\sqrt{x}\sum_{t=0}^{\infty}
\left(\sum_{i=0}^{t}(-1)^i\binom{k}{i}\binom{2n-2k-1}{2(t-i)+1}\right)x^t.
\]
Now $\displaystyle \binom{k}{i}=0$ for $i>k$, while $\displaystyle \binom{2n-2k-1}{2(t-i)+1}=0$ for $t>n-1$, hence
\[
P_k(\sqrt{x})Q_{n,k}(\sqrt{x})
=
2\sqrt{x}\sum_{t=0}^{n-1}
\left(\sum_{i=0}^{k}(-1)^i\binom{k}{i}\binom{2n-2k-1}{2(t-i)+1}\right)x^t.
\]
Replacing $x$ by $x^2$ gives the claim.
\end{proof}

\begin{lemma}\label[lemma]{lemma:wood_identity_for_type_B}
Let $m\in\mathbb{N}_0$ and $|z|<1$. Then
\[
\frac{\Im\!\bigl(\operatorname{Li}_{-m}(iz)\bigr)}{z}
=\frac{1}{(1+z^2)^{m+1}}\sum_{r=0}^{m}\EulerB{m}{r}\,(-z^2)^r.
\]
\end{lemma}

\begin{proof}
For $|w|<1$, the polylogarithm is given by $\displaystyle \operatorname{Li}_{-m}(w)=\sum_{k=1}^{\infty} k^m w^k$. Thus, with $w=iz$,
$\displaystyle \operatorname{Li}_{-m}(iz)=\sum_{k=1}^{\infty} k^m i^k z^k$, and taking imaginary parts yields
$\displaystyle \Im\!\bigl(\operatorname{Li}_{-m}(iz)\bigr)=\sum_{j=0}^{\infty}\bigl((4j+1)^m z^{4j+1}-(4j+3)^m z^{4j+3}\bigr)$.
Dividing by $z$ and setting $\displaystyle x:=-z^2$, so that $\displaystyle z^{4j}=x^{2j}$ and $\displaystyle z^{4j+2}=(-z^2)^{2j+1}=x^{2j+1}$, we obtain
\begin{equation}
\frac{\Im\!\bigl(\operatorname{Li}_{-m}(iz)\bigr)}{z}
=\sum_{j=0}^{\infty}(4j+1)^m x^{2j}+\sum_{j=0}^{\infty}(4j+3)^m x^{2j+1}
=\sum_{k=0}^{\infty}(2k+1)^m x^k,
\qquad x=-z^2.
\label{eq:log_expansion}
\end{equation}

Now use the Worpitzky-type identity for type-$B$ Eulerian numbers,
$\displaystyle \sum_{\ell=0}^{m}\binom{m+k-\ell}{m}\EulerB{m}{\ell}=(2k+1)^m$ for $k\ge0$.
Multiplying by $\displaystyle x^k$ and summing over $k\ge0$ gives
$\displaystyle \sum_{k=0}^{\infty}(2k+1)^m x^k
=\sum_{\ell=0}^{m}\EulerB{m}{\ell}\sum_{k=0}^{\infty}\binom{m+k-\ell}{m}x^k$.
In the inner sum, writing $\displaystyle k=\ell+j$ gives
$\displaystyle \sum_{k=0}^{\infty}\binom{m+k-\ell}{m}x^k
=\sum_{j=0}^{\infty}\binom{m+j}{m}x^{j+\ell}
=x^\ell\sum_{j=0}^{\infty}\binom{m+j}{m}x^j$,
and by the binomial series,
$\displaystyle \sum_{j=0}^{\infty}\binom{m+j}{m}x^j=\frac{1}{(1-x)^{m+1}}$ for $|x|<1$.
Hence
\begin{equation}
\sum_{k=0}^{\infty}(2k+1)^m x^k
=\frac{1}{(1-x)^{m+1}}\sum_{\ell=0}^{m}\EulerB{m}{\ell}x^\ell,
\qquad |x|<1.
\label{eq:final}
\end{equation}

Combining \eqref{eq:log_expansion} and \eqref{eq:final} with $\displaystyle x=-z^2$ yields
$\displaystyle \frac{\Im\!\bigl(\operatorname{Li}_{-m}(iz)\bigr)}{z}
=\frac{1}{(1+z^2)^{m+1}}\sum_{\ell=0}^{m}\EulerB{m}{\ell}(-z^2)^\ell$,
which proves the claim.
\end{proof}

\section{Limiting Ratios of Series Related to Polylogarithms and Eulerian Polynomials}

\begin{lemma}\label[lemma]{lemma:affine-series-ratio}
Let \(\alpha,\beta>0\), and let
\[
\mathbb D:=\{x\in\mathbb C:\ |x|<1\},\qquad \mathbb D^\ast:=\mathbb D\setminus(-1,0].
\]
For \(m\in\mathbb N\) and \(x\in\mathbb D^\ast\), define
\[
T_m^{(\alpha,\beta)}(x):=\sum_{n\ge0}(\alpha n+\beta)^m x^n.
\]
Then
\[
\frac1m\,\frac{T_{m+1}^{(\alpha,\beta)}(x)}{T_m^{(\alpha,\beta)}(x)}
\longrightarrow
\frac{\alpha}{-\log x},
\qquad m\to\infty,
\]
locally uniformly on \(\mathbb D^\ast\), where \(\log\) denotes the principal branch.
\end{lemma}

\begin{proof}
Fix a compact set \(K\subset\mathbb D^\ast\). We prove the convergence uniformly on \(K\).

For \(m\in\mathbb N\),
\[
T_m^{(\alpha,\beta)}(x)=\sum_{n\ge0}(\alpha n+\beta)^m x^n.
\]
By Cauchy's integral formula for derivatives, for any \(r>0\) and \(n\ge0\),
\[
(\alpha n+\beta)^m
=
\left.\frac{d^m}{dz^m}e^{(\alpha n+\beta)z}\right|_{z=0}
=
\frac{m!}{2\pi i}\oint_{|z|=r}\frac{e^{(\alpha n+\beta)z}}{z^{m+1}}\,dz.
\]
Hence
\[
T_m^{(\alpha,\beta)}(x)
=
\frac{m!}{2\pi i}\oint_{|z|=r}\frac{1}{z^{m+1}}
\sum_{n\ge0}x^n e^{(\alpha n+\beta)z}\,dz.
\]

Since \(K\subset\mathbb D\) is compact, there exists \(M<1\) such that \(|x|\le M\) for all \(x\in K\).
Choose \(r>0\) such that
\[
Me^{\alpha r}<1.
\]
Then, for \(|z|=r\) and \(x\in K\),
\[
|x e^{\alpha z}|
\le |x|e^{\alpha|z|}
\le Me^{\alpha r}
<1,
\]
so the series \(\sum_{n\ge0}x^n e^{(\alpha n+\beta)z}\) converges absolutely and uniformly on
\(K\times\{|z|=r\}\). Therefore we may interchange sum and integral, obtaining
\[
\sum_{n\ge0}x^n e^{(\alpha n+\beta)z}
=
e^{\beta z}\sum_{n\ge0}(x e^{\alpha z})^n
=
\frac{e^{\beta z}}{1-x e^{\alpha z}}.
\]
Thus
\[
T_m^{(\alpha,\beta)}(x)
=
\frac{m!}{2\pi i}\oint_{|z|=r}\frac{g_x(z)}{z^{m+1}}\,dz,
\qquad
g_x(z):=\frac{e^{\beta z}}{1-x e^{\alpha z}}.
\]

The poles of \(g_x\) are the solutions of \(1-x e^{\alpha z}=0\), namely
\[
z_k(x)=\frac{-\log x+2\pi i k}{\alpha},
\qquad k\in\mathbb Z.
\]
Write
\[
\log x=\log|x|+i\Arg(x),
\qquad -\pi<\Arg(x)<\pi,
\]
and set
\[
z_0(x):=\frac{-\log x}{\alpha}.
\]
Since \(x\in\mathbb D^\ast\), we have \(-\log x\neq0\). Moreover, because \(K\subset\mathbb D^\ast\) is compact,
there exists \(b<\pi\) such that
\[
|\Arg(x)|\le b
\qquad (x\in K).
\]
Hence, for \(k\neq0\),
\[
|2\pi k-\Arg(x)|\ge 2\pi-b>b\ge |\Arg(x)|.
\]
It follows that
\[
|z_k(x)|^2
=
\frac{(-\log|x|)^2+(2\pi k-\Arg(x))^2}{\alpha^2}
>
\frac{(-\log|x|)^2+\Arg(x)^2}{\alpha^2}
=
|z_0(x)|^2.
\]
Thus \(z_0(x)\) is the unique pole of \(g_x\) of minimal modulus, uniformly for \(x\in K\).

Set
\[
d_0:=\sup_{x\in K}|z_0(x)|,
\qquad
d_1:=\inf_{\substack{x\in K\\ k\neq0}}|z_k(x)|.
\]
Then \(0<d_0<d_1\). Choose \(R>0\) such that
\[
d_0<R<d_1.
\]
For each \(x\in K\), the function \(g_x\) is meromorphic on and inside \(|z|=R\), with exactly one pole
in the closed annulus \(r\le |z|\le R\), namely \(z_0(x)\). By the residue theorem,
\[
\oint_{|z|=r}\frac{g_x(z)}{z^{m+1}}\,dz
=
\oint_{|z|=R}\frac{g_x(z)}{z^{m+1}}\,dz
-
2\pi i\,\operatorname{Res}\!\left(\frac{g_x(z)}{z^{m+1}},z_0(x)\right).
\]

The pole at \(z_0(x)\) is simple, and its residue is
\[
\operatorname{Res}(g_x,z_0(x))
=
\frac{e^{\beta z_0(x)}}{(1-x e^{\alpha z})'|_{z=z_0(x)}}
=
\frac{e^{\beta z_0(x)}}{-\alpha x e^{\alpha z_0(x)}}
=
-\frac{e^{\beta z_0(x)}}{\alpha},
\]
since \(x e^{\alpha z_0(x)}=1\). Therefore,
\[
\operatorname{Res}\!\left(\frac{g_x(z)}{z^{m+1}},z_0(x)\right)
=
-\frac{e^{\beta z_0(x)}}{\alpha\,z_0(x)^{m+1}},
\]
and hence
\[
T_m^{(\alpha,\beta)}(x)
=
\frac{m!}{2\pi i}\oint_{|z|=R}\frac{g_x(z)}{z^{m+1}}\,dz
+
\frac{m!e^{\beta z_0(x)}}{\alpha\,z_0(x)^{m+1}}.
\]

Now \(1-xe^{\alpha z}\neq0\) on \(K\times\{|z|=R\}\), and this set is compact. Hence
\[
|g_x(z)|\le C
\qquad (x\in K,\ |z|=R)
\]
for some constant \(C>0\). Thus
\[
\left|
\frac{m!}{2\pi i}\oint_{|z|=R}\frac{g_x(z)}{z^{m+1}}\,dz
\right|
\le
\frac{m!}{2\pi}(2\pi R)\frac{C}{R^{m+1}}
=
m!CR^{-m},
\]
uniformly for \(x\in K\).

Since \(K\) is compact and \(0\notin K\), there exists \(c>0\) such that
\[
|z_0(x)|\ge c
\qquad (x\in K).
\]
Therefore
\[
T_m^{(\alpha,\beta)}(x)
=
\frac{m!e^{\beta z_0(x)}}{\alpha\,z_0(x)^{m+1}}
\left(
1+O\!\left(\left(\frac{d_0}{R}\right)^m\right)
\right),
\qquad m\to\infty,
\]
uniformly for \(x\in K\). Applying the same estimate with \(m+1\) in place of \(m\), we get
\[
T_{m+1}^{(\alpha,\beta)}(x)
=
\frac{(m+1)!e^{\beta z_0(x)}}{\alpha\,z_0(x)^{m+2}}
\left(
1+O\!\left(\left(\frac{d_0}{R}\right)^{m+1}\right)
\right),
\]
uniformly on \(K\). Since \(d_0/R<1\), it follows that
\[
\frac{T_{m+1}^{(\alpha,\beta)}(x)}{T_m^{(\alpha,\beta)}(x)}
=
\frac{m+1}{z_0(x)}\bigl(1+o(1)\bigr),
\]
uniformly for \(x\in K\). Dividing by \(m\), we obtain
\[
\frac1m\,\frac{T_{m+1}^{(\alpha,\beta)}(x)}{T_m^{(\alpha,\beta)}(x)}
\longrightarrow
\frac1{z_0(x)}
=
\frac{\alpha}{-\log x},
\]
uniformly on \(K\). Since \(K\subset\mathbb D^\ast\) was arbitrary, the convergence is locally uniform on \(\mathbb D^\ast\).
This proves the claim.
\end{proof}

\begin{corollary}\label[corollary]{cor:A-B-ratio}
Let 
\(
\mathbb D:=\{x\in\mathbb C:\ |x|<1\},\qquad \mathbb D^\ast:=\mathbb D\setminus(-1,0].
\)
For every \(x\in\mathbb D^\ast\), let \(A_m(x)\) and \(B_m(x)\) denote the Eulerian polynomials of type \(A\) and type \(B\), respectively. Then
\[
\frac1m\,\frac{A_{m+1}(x)}{A_m(x)}
\longrightarrow
\frac{1-x}{-\log x},
\qquad
\frac1m\,\frac{B_{m+1}(x)}{B_m(x)}
\longrightarrow
\frac{2(1-x)}{-\log x},
\]
locally uniformly on \(\mathbb D^\ast\), where \(\log\) denotes the principal branch.
\end{corollary}

\begin{proof}
It is classical that
\[
\sum_{k=0}^\infty (k+1)^m x^k
=
\frac{A_m(x)}{(1-x)^{m+1}},
\qquad |x|<1,
\]
\[
\sum_{k=0}^\infty (2k+1)^m x^k
=
\frac{B_m(x)}{(1-x)^{m+1}},
\qquad |x|<1.
\]
Thus, for \(x\in\mathbb D^\ast\),
\[
A_m(x)=(1-x)^{m+1}\sum_{k=0}^\infty (k+1)^m x^k,
\qquad
B_m(x)=(1-x)^{m+1}\sum_{k=0}^\infty (2k+1)^m x^k.
\]
Hence
\[
\frac{A_{m+1}(x)}{A_m(x)}
=
(1-x)\,
\frac{\displaystyle\sum_{k=0}^\infty (k+1)^{m+1}x^k}
     {\displaystyle\sum_{k=0}^\infty (k+1)^m x^k},
\]
\[
\frac{B_{m+1}(x)}{B_m(x)}
=
(1-x)\,
\frac{\displaystyle\sum_{k=0}^\infty (2k+1)^{m+1}x^k}
     {\displaystyle\sum_{k=0}^\infty (2k+1)^m x^k}.
\]

Applying \cref{lemma:affine-series-ratio} with \((\alpha,\beta)=(1,1)\) and \((\alpha,\beta)=(2,1)\), respectively, we obtain
\[
\frac1m\,\frac{A_{m+1}(x)}{A_m(x)}
\longrightarrow
(1-x)\frac{1}{-\log x}
=
\frac{1-x}{-\log x},
\]
\[
\frac1m\,\frac{B_{m+1}(x)}{B_m(x)}
\longrightarrow
(1-x)\frac{2}{-\log x}
=
\frac{2(1-x)}{-\log x}.
\]
The convergence is locally uniform on \(\mathbb D^\ast\) by the locally uniform convergence in \cref{lemma:affine-series-ratio}.
This proves the claim.
\end{proof}

\section{Convergence of Extremal zeros from endpoint ratios}
\begin{theorem}[Smallest zeros from endpoint ratios at a left endpoint $a$]\label[theorem]{theorem:bounding_smallest_root_a}
Let $(P_n)_{n\ge 1}$ be a sequence of real polynomials.
Fix a real number $a$. Assume that:
\begin{enumerate}
\item $\deg P_n\xrightarrow[n\to\infty]{}\+\infty$,
\item $P_n$ has only real zeros, counted with multiplicity, all contained in the open interval $(a,+\infty)$.
Denote them by: 
\(
a < r_{n,1} \leq r_{n,2} \leq \cdots \leq r_{n,\deg P_n}.
\)
\item Writing $\operatorname{lc}(P_n)$ for the leading coefficient of $P_n$, the \emph{normalized endpoint ratio} satisfies
\[
\left|\frac{P_n(a)}{\operatorname{lc}(P_n)}\right|^{1/\deg P_n}\xrightarrow[n\to\infty]{}0.
\]
\end{enumerate}
Then the smallest zero of $P_n$ converges to $a$; more precisely,
\(\displaystyle
\lim_{n\to\infty} r_{n,1}=a.
\)
Moreover, for every $n$ one has the quantitative bound:
\(\displaystyle
r_{n,1}-a\le \left|\frac{P_n(a)}{\operatorname{lc}(P_n)}\right|^{1/\deg P_n}.
\)
\end{theorem}

\begin{proof}
Since all zeros are real and lie in $(a,+\infty)$, we may factor
\[
P_n(x)=\operatorname{lc}(P_n)\prod_{j=1}^{\deg P_n}(x-r_{n,j}).
\]
Evaluating at $x=a$ gives
\[
\frac{P_n(a)}{\operatorname{lc}(P_n)}
=\prod_{j=1}^{\deg P_n}(a-r_{n,j})
=(-1)^{\deg P_n}\prod_{j=1}^{\deg P_n}(r_{n,j}-a).
\]
Taking absolute values,
\[
\left|\frac{P_n(a)}{\operatorname{lc}(P_n)}\right|
=\prod_{j=1}^{\deg P_n}(r_{n,j}-a).
\]
All factors are positive since \(a < r_{n,j} \quad \forall j \in \left\{1, 2, \cdots, \deg P_n\right\}\), hence
\[
\prod_{j=1}^{\deg P_n}(r_{n,j}-a)\ \ge\ \bigl(\min_{1\le j\le \deg P_n}(r_{n,j}-a)\bigr)^{\deg P_n}
=(r_{n,1}-a)^{\deg P_n}.
\]
Therefore,
\[
r_{n,1}-a
\le \left(\prod_{j=1}^{\deg P_n}(r_{n,j}-a)\right)^{1/\deg P_n}
=\left|\frac{P_n(a)}{\operatorname{lc}(P_n)}\right|^{1/\deg P_n}.
\]
By assumption the right-hand side tends to $0$, hence $r_{n,1}\to a$.
The displayed inequality also yields the quantitative bound.
\end{proof}

\begin{theorem}[Largest zeros from endpoint ratios]\label[theorem]{theorem:bounding_largest_root_a}
Let $(P_n)_{n\ge 1}$ be a sequence of real polynomials. Fix an arbitrary real number $a$.
Assume that:
\begin{enumerate}
\item $\deg P_n\xrightarrow[n\to\infty]{}\+\infty$
\item $P_n$ has only real zeros, counted with multiplicity, all contained in the open interval $(-\infty,a)$.
Denote them by:
\(
r_{n,1} \leq r_{n,2} \leq \cdots \leq r_{n,\deg P_n} < a.
\)
\item Writing $\operatorname{lc}(P_n)$ for the leading coefficient of $P_n$, the \emph{normalized endpoint ratio} satisfies
\[
\left|\frac{P_n(a)}{\operatorname{lc}(P_n)}\right|^{1/\deg P_n}\xrightarrow[n\to\infty]{}0.
\]
\end{enumerate}
Then the largest zero of $P_n$ converges to $a$; more precisely,
\(\displaystyle
\lim_{n\to\infty} r_{n,\deg P_n}=a.
\)
Moreover, for every $n$ one has the quantitative bound:
\(\displaystyle
a-r_{n,\deg P_n}\le \left|\frac{P_n(a)}{\operatorname{lc}(P_n)}\right|^{1/\deg P_n}.
\)
\end{theorem}

\begin{proof}
Since all zeros lie in $(-\infty,a)$, the polynomial $P_n$ admits the factorization
\[
P_n(x)=\operatorname{lc}(P_n)\prod_{j=1}^{\deg P_n}(x-r_{n,j}).
\]

Evaluating at $x=a$ yields
\[
\frac{P_n(a)}{\operatorname{lc}(P_n)}
=\prod_{j=1}^{\deg P_n}(a-r_{n,j}).
\]
All factors are positive since \(r_{n,j} < a \quad \forall j \in \left\{1, 2, \cdots, \deg P_n\right\}\). Hence
\[
\prod_{j=1}^{\deg P_n}(a-r_{n,j}) \;\ge\; \bigl(\min_{1\le j\le \deg P_n}(a-r_{n,j})\bigr)^{\deg P_n}
= (a-r_{n,\deg P_n})^{\deg P_n},
\]
and therefore
\[
a-r_{n,\deg P_n}\le \left(\prod_{j=1}^{\deg P_n}(a-r_{n,j})\right)^{1/\deg P_n}
=\left|\frac{P_n(a)}{\operatorname{lc}(P_n)}\right|^{1/\deg P_n}.
\]
By assumption, the right-hand side tends to $0$, so $a-r_{n,\deg P_n}\to 0$, i.e. $r_{n,\deg P_n}\to a$. The displayed inequalities also prove the stated quantitative bounds. This completes the proof.
\end{proof}

\vspace{0.3cm}

\part{Existence and explicit formulae for the polynomials}\label[part]{part:existence_and_closed_forms}

\vspace{0.3cm}

\section{Existence, expansion form and formula of coefficients}

\begin{theorem}\label{theorem:explicit_polynomials}
Let $n\in\mathbb N$ and $t\in\mathbb N_0$. We adopt the standard convention that
\[
\binom{m}{\ell}=0 \qquad\text{whenever }\ell<0\text{ or }\ell>m.
\]
Set
\[
\mathcal C^{(B)}_{n,t}
:=
\sum_{k=0}^{n-1}\left\langle {2n-1 \atop k}\right\rangle^{B}(-1)^{k}
\sum_{i=0}^{k}(-1)^i\binom{k}{i}\binom{2n-2k-1}{2t-2i+1},
\qquad 0\le t\le n-1,
\]
\[
\mathcal C^{(A)}_{n,t}
:=
\sum_{k=0}^{n-1}\left\langle {2n \atop k}\right\rangle(-1)^{k}
\sum_{i=0}^{k}(-1)^i\binom{k}{i}\binom{2n-2k-1}{2t-2i+1},
\qquad 0\le t\le n-1.
\]
Consider the even polynomials $\Xi_n,\Lambda_n\in\mathbb Q[x]$ of degree at most $2n-2$ given by
\begin{equation}\label{eq:XiLambda-explicit}
\Xi_n(x)=\frac{(-1)^{n+1}}{2^{4n-2}(2n-1)!}\sum_{t=0}^{n-1}\mathcal C^{(B)}_{n,t}\,x^{2t},
\qquad
\Lambda_n(x)=\frac{2(-1)^{n+1}}{(2^{2n+1}-1)(2n)!}\sum_{t=0}^{n-1}\mathcal C^{(A)}_{n,t}\,x^{2t}.
\end{equation}
Then the following integral representations hold true
\begin{equation}\label{eq:beta-zeta-reps}
\frac{\beta(2n)}{\pi^{2n-1}}
=
\int_0^1 \frac{x\,\Xi_n(x)}{\sqrt{1-x^2}\,\operatorname{arctanh}x}\,dx,
\qquad
\frac{\zeta(2n+1)}{\pi^{2n}}
=
\int_0^1 \frac{x\,\Lambda_n(x)}{\operatorname{arctanh}x}\,dx.
\end{equation}
\end{theorem}

\begin{proof}
We start from the two identities \cite{TallaWaffo2025arxiv2511.02843}
\begin{equation}\label{eq:start}
\begin{cases}
\displaystyle
\int_{0}^{1}
\frac{\Im\!\big(\operatorname{Li}_{-2n}(ix)\big)}{x}\,
\ln\ln\frac{1}{x}\,dx
=
(-1)^{n+1}
\frac{2^{2n-1}(2n-1)!}{\pi^{2n-1}}
\,\beta(2n),
\\[2.6em]
\displaystyle
\int_{0}^{1}
\frac{\operatorname{Li}_{-2n-1}(-x^{2})}{x}\,
\ln\ln\frac{1}{x}\,dx
=
(-1)^{n}
\Bigl(1-2^{-(2n+1)}\Bigr)
\frac{(2n)!}{2\pi^{2n}}
\,\zeta(2n+1).
\end{cases}
\end{equation}

We perform integration by parts, differentiating $\ln\ln(1/x)$. We use
\[
\frac{d}{dx}\,\operatorname{Li}_s(a x)=\frac{1}{x}\,\operatorname{Li}_{s-1}(a x)\qquad (a\neq 0),
\]
and, for $z=-x^2$,
\[
\frac{d}{dx}\,\operatorname{Li}_{-2n}(-x^2)=\frac{2}{x}\,\operatorname{Li}_{-2n-1}(-x^2).
\]
The boundary terms vanish. Hence
\begin{equation}\label{eq:ibp}
\begin{cases}
\displaystyle
\int_{0}^{1}
\frac{\Im\!\big(\operatorname{Li}_{-2n}(ix)\big)}{x}\,
\ln\ln\frac{1}{x}\,dx
=
-\displaystyle
\int_{0}^{1}
\frac{\Im\!\big(\operatorname{Li}_{-2n+1}(ix)\big)}{x\ln x}\,dx,
\\[2.0em]
\displaystyle
\int_{0}^{1}
\frac{\operatorname{Li}_{-2n-1}(-x^{2})}{x}\,
\ln\ln\frac{1}{x}\,dx
=
-\displaystyle
\frac{1}{2}
\int_{0}^{1}
\frac{\operatorname{Li}_{-2n}(-x^{2})}{x\ln x}\,dx.
\end{cases}
\end{equation}

Next we insert the rational forms of negative-index polylogarithms in terms of Eulerian numbers of type A and B \cite{Wood1992Polylog}, \cref{lemma:wood_identity_for_type_B}:
\begin{equation}\label{eq:polylog-eulerian}
\operatorname{Li}_{-m}(z)=\frac{1}{(1-z)^{m+1}}\sum_{r=0}^{m-1}\left\langle{m\atop r}\right\rangle z^{m-r},
\qquad
\frac{\Im(\operatorname{Li}_{-m}(iz))}{z}=\frac{1}{(1+z^2)^{m+1}}\sum_{r=0}^{m}\left\langle{m\atop r}\right\rangle^{B}(-z^2)^r.
\end{equation}
Applying \eqref{eq:polylog-eulerian} with $m=2n-1$ in the first integral and $m=2n$ in the second one gives
\[
\begin{cases}
\displaystyle
\int_{0}^{1}
\frac{\Im\!\big(\operatorname{Li}_{-2n}(ix)\big)}{x}\,
\ln\ln\frac{1}{x}\,dx
=
-\displaystyle
\int_{0}^{1}
\frac{1}{\ln x}\,
\frac{1}{(1+x^{2})^{2n}}
\sum_{r=0}^{2n-1}
\left\langle{2n-1\atop r}\right\rangle^{B}
(-x^{2})^{r}\,dx,
\\[2.4em]
\displaystyle
\int_{0}^{1}
\frac{\operatorname{Li}_{-2n-1}(-x^{2})}{x}\,
\ln\ln\frac{1}{x}\,dx
=
-\displaystyle
\frac{1}{2}
\int_{0}^{1}
\frac{1}{x\ln x}\,
\frac{1}{(1+x^{2})^{2n+1}}
\sum_{r=0}^{2n-1}
\left\langle{2n\atop r}\right\rangle
(-x^{2})^{2n-r}\,dx.
\end{cases}
\]

Split each finite sum at $r=n$, re-index and reverse the second half, and use the Eulerian symmetries \cite{FulmanLeeKimPetersen2021}
\[
\left\langle{m\atop r}\right\rangle=\left\langle{m\atop m-1-r}\right\rangle,
\qquad
\left\langle{m\atop r}\right\rangle^{B}=\left\langle{m\atop m-r}\right\rangle^{B}.
\]
Thus
\[
\begin{cases}
\displaystyle
\int_{0}^{1}
\frac{\Im\!\big(\operatorname{Li}_{-2n}(ix)\big)}{x}\,
\ln\ln\frac{1}{x}\,dx
=
-\displaystyle
\int_{0}^{1}
\frac{1}{\ln x}\,
\frac{1}{(1+x^{2})^{2n}}
\sum_{k=0}^{n-1}
\left\langle{2n-1\atop k}\right\rangle^{B}
\left\{(-x^{2})^{k}+(-x^{2})^{2n-1-k}\right\}\,dx,
\\[2.4em]
\displaystyle
\int_{0}^{1}
\frac{\operatorname{Li}_{-2n-1}(-x^{2})}{x}\,
\ln\ln\frac{1}{x}\,dx
=
-\displaystyle
\frac{1}{2}
\int_{0}^{1}
\frac{1}{x\ln x}\,
\frac{1}{(1+x^{2})^{2n+1}}
\sum_{k=0}^{n-1}
\left\langle{2n\atop k}\right\rangle
\left\{(-x^{2})^{2n-k}+(-x^{2})^{k+1}\right\}\,dx.
\end{cases}
\]

Bearing in mind that
\(\displaystyle
\operatorname{arctanh}x=\frac12\log\frac{1+x}{1-x},
\)
the substitution
\(\displaystyle
u=\frac{1-x^2}{1+x^2}
\) yields
\begin{equation}\label{eq:arctanh-stage}
\begin{cases}
\displaystyle
\int_{0}^{1}
\frac{\Im\!\big(\operatorname{Li}_{-2n}(ix)\big)}{x}\,
\ln\ln\frac{1}{x}\,dx
=
\displaystyle
\frac{1}{2^{2n}}
\sum_{k=0}^{n-1}
\left\langle {2n-1 \atop k}\right\rangle^{B}(-1)^k
\int_{0}^{1}
\frac{E_{n,k}(u)}{\operatorname{arctanh}u\,\sqrt{1-u^{2}}}\,du,
\\[2.6em]
\displaystyle
\int_{0}^{1}
\frac{\operatorname{Li}_{-2n-1}(-x^{2})}{x}\,
\ln\ln\frac{1}{x}\,dx
=
-\displaystyle
\frac{1}{2^{2n+2}}
\sum_{k=0}^{n-1}
\left\langle {2n \atop k}\right\rangle(-1)^k
\int_{0}^{1}
\frac{E_{n,k}(u)}{\operatorname{arctanh}u}\,du,
\end{cases}
\end{equation}
where
\[
E_{n,k}(u):=(1-u^2)^k\Big((1+u)^{2n-2k-1}-(1-u)^{2n-2k-1}\Big).
\]

By \cref{lemma:Enk}, for $0\le k\le n-1$,
\[
E_{n,k}(u)
=
2u\sum_{t=0}^{n-1}
\left(
\sum_{i=0}^{k}
(-1)^i
\binom{k}{i}
\binom{2n-2k-1}{2t-2i+1}
\right)
u^{2t}.
\]
Since the sums in \eqref{eq:arctanh-stage} are finite, we may reorder them freely and move summations outside the integral:
\begin{equation}\label{eq:series-stage}
\begin{cases}
\displaystyle
\int_{0}^{1}
\frac{\Im\!\big(\operatorname{Li}_{-2n}(ix)\big)}{x}\,
\ln\ln\frac{1}{x}\,dx
=
\displaystyle
\frac{1}{2^{2n}}
\sum_{t=0}^{n-1}
\mathcal C^{(B)}_{n,t}
\int_{0}^{1}
\frac{2u^{2t+1}}{\operatorname{arctanh}u\,\sqrt{1-u^{2}}}\,du,
\\[2.6em]
\displaystyle
\int_{0}^{1}
\frac{\operatorname{Li}_{-2n-1}(-x^{2})}{x}\,
\ln\ln\frac{1}{x}\,dx
=
-\displaystyle
\frac{1}{2^{2n+2}}
\sum_{t=0}^{n-1}
\mathcal C^{(A)}_{n,t}
\int_{0}^{1}
\frac{2u^{2t+1}}{\operatorname{arctanh}u}\,du.
\end{cases}
\end{equation}

Finally, define $\Xi_n$ and $\Lambda_n$ by \eqref{eq:XiLambda-explicit}. Then
\[
\int_0^1 \frac{u\,\Xi_n(u)}{\sqrt{1-u^2}\,\operatorname{arctanh}u}\,du
=
\frac{(-1)^{n+1}}{2^{4n-2}(2n-1)!}
\sum_{t=0}^{n-1}\mathcal C^{(B)}_{n,t}
\int_0^1 \frac{u^{2t+1}}{\sqrt{1-u^2}\,\operatorname{arctanh}u}\,du,
\]
\[
\int_0^1 \frac{u\,\Lambda_n(u)}{\operatorname{arctanh}u}\,du
=
\frac{2(-1)^{n+1}}{(2^{2n+1}-1)(2n)!}
\sum_{t=0}^{n-1}\mathcal C^{(A)}_{n,t}
\int_0^1 \frac{u^{2t+1}}{\operatorname{arctanh}u}\,du.
\]
Comparing with \eqref{eq:series-stage} and using \eqref{eq:start}, we recover exactly the integral representations \eqref{eq:beta-zeta-reps}. This completes the proof.
\end{proof}
For illustrative examples of these polynomials, see \cite{TallaWaffo2025arxiv2511.02843}.

\section{Expression depending on Eulerian polynomials}

\begin{proposition}\label[proposition]{prop:eulerian_polynomial_form}
Let $n\in\mathbb N$ and let $\Xi_n,\Lambda_n\in\mathbb Q[x]$ be the even polynomials of
\cref{theorem:explicit_polynomials}, given explicitly by \eqref{eq:XiLambda-explicit}.
Let $E_{n,k}$ be defined as in \cref{lemma:Enk}. Define the Eulerian polynomials of type~B and type~A by
\[
B_{n}(z):=\sum_{k=0}^{n}\left\langle{n\atop k}\right\rangle^{B}z^k,
\qquad
A_{n}(z):=\sum_{k=0}^{n-1}\left\langle{n\atop k}\right\rangle z^k.
\]
Then, for every $x\in\mathbb C\setminus\{0,-1\}$,
\begin{equation}\label{eq:XiLambda-eulerian-final}
\begin{cases}
\displaystyle
\Xi_n(x)
=
\frac{(-1)^{n+1}}{2^{4n-1}(2n-1)!}\,
\frac{(1+x)^{2n-1}}{x}\,
B_{2n-1}\!\left(-\,\dfrac{1-x}{1+x}\right),\\[1.2em]
\displaystyle
\Lambda_n(x)
=
\frac{(-1)^{n+1}}{(2^{2n+1}-1)(2n)!}\,
\frac{(1+x)^{2n-1}}{x}\,
A_{2n}\!\left(-\,\dfrac{1-x}{1+x}\right).
\end{cases}
\end{equation}
\end{proposition}

\begin{proof}
We first rewrite $\Xi_n$ and $\Lambda_n$ in terms of the auxiliary functions
$E_{n,k}$, and then evaluate these at the Möbius-transformed argument
\[
y=\frac{1-x}{1+x}.
\]

By \cref{lemma:Enk} we have, for $0\le k\le n-1$,
\[
E_{n,k}(x)
=
2x\sum_{t=0}^{n-1}
\biggl(
\sum_{i=0}^{k}
(-1)^i\binom{k}{i}
\binom{2n-2k-1}{\,2t-2i+1\,}
\biggr)x^{2t}.
\]
Comparing with the definition of $\mathcal C^{(B)}_{n,t}$ and $\mathcal C^{(A)}_{n,t}$
in \eqref{eq:XiLambda-explicit}, and using that all sums are finite,
we may interchange the order of summation and obtain the alternative
expansion
\[
\begin{cases}
\displaystyle
\Xi_n(x)
=
\frac{(-1)^{n+1}}{2^{4n-1}(2n-1)!\,x}
\sum_{k=0}^{n-1}
\left\langle{2n-1\atop k}\right\rangle^{B}
(-1)^k\,E_{n,k}(x),\\[1.2em]
\displaystyle
\Lambda_n(x)
=
\frac{(-1)^{n+1}}{(2^{2n+1}-1)(2n)!\,x}
\sum_{k=0}^{n-1}
\left\langle{2n\atop k}\right\rangle
(-1)^k\,E_{n,k}(x).
\end{cases}
\]

Bearing in mind that $E_{n,k}(x):=(1-x^2)^k\Big((1+x)^{2n-2k-1}-(1-x)^{2n-2k-1}\Big)$, we prove this form
\[
E_{n,k}\!\left(\frac{1-x}{1+x}\right)
=
\frac{2^{2n-1}}{(1+x)^{2n-1}}\,
x^k\bigl(1-x^{2n-2k-1}\bigr).
\]

This gives
\[
\begin{cases}
\displaystyle
\Xi_n\!\left(\frac{1-x}{1+x}\right)
=
\frac{(-1)^{n+1}}{2^{2n}(2n-1)!\,(1-x)(1+x)^{2n-2}}
\sum_{k=0}^{n-1}
\left\langle{2n-1\atop k}\right\rangle^{B}
(-1)^k x^k\bigl(1-x^{2n-2k-1}\bigr),\\[1.2em]
\displaystyle
\Lambda_n\!\left(\frac{1-x}{1+x}\right)
=
\frac{(-1)^{n+1}2^{2n-1}}{(2^{2n+1}-1)(2n)!\,(1-x)(1+x)^{2n-2}}
\sum_{k=0}^{n-1}
\left\langle{2n\atop k}\right\rangle
(-1)^k x^k\bigl(1-x^{2n-2k-1}\bigr).
\end{cases}
\]
Using
\[
x^k\bigl(1-x^{2n-2k-1}\bigr)=x^k-x^{2n-1-k},
\]
we split each finite sum into two parts and then re-index the second half
via $j=2n-1-k$, exactly as in the proof of \cref{theorem:explicit_polynomials}.
Employing the Eulerian symmetries
\[
\left\langle{m\atop r}\right\rangle
=
\left\langle{m\atop m-1-r}\right\rangle,
\qquad
\left\langle{m\atop r}\right\rangle^{B}
=
\left\langle{m\atop m-r}\right\rangle^{B},
\]
we obtain the compact form
\[
\begin{cases}
\displaystyle
\Xi_n\!\left(\frac{1-x}{1+x}\right)
=
\frac{(-1)^{n+1}}{2^{2n}(2n-1)!\,(1-x)(1+x)^{2n-2}}
\sum_{k=0}^{2n-1}
\left\langle{2n-1\atop k}\right\rangle^{B}
(-1)^k x^{k},\\[1.2em]
\displaystyle
\Lambda_n\!\left(\frac{1-x}{1+x}\right)
=
\frac{(-1)^{n+1}2^{2n-1}}{(2^{2n+1}-1)(2n)!\,(1-x)(1+x)^{2n-2}}
\sum_{k=0}^{2n-1}
\left\langle{2n\atop k}\right\rangle
(-1)^k x^{k}.
\end{cases}
\]

By the definition of the Eulerian polynomials, this becomes
\[
\begin{cases}
\displaystyle
\Xi_n\!\left(\frac{1-x}{1+x}\right)
=
\frac{(-1)^{n+1}}{2^{2n}(2n-1)!\,(1-x)(1+x)^{2n-2}}\,
B_{2n-1}(-x),\\[1.2em]
\displaystyle
\Lambda_n\!\left(\frac{1-x}{1+x}\right)
=
\frac{(-1)^{n+1}2^{2n-1}}{(2^{2n+1}-1)(2n)!\,(1-x)(1+x)^{2n-2}}\,
A_{2n}(-x).
\end{cases}
\]

Substituting $x \to (1-x)/(1+x)$, we arrive at
\[
\begin{cases}
\displaystyle
\Xi_n(x)
=
\frac{(-1)^{n+1}}{2^{4n-1}(2n-1)!}\,
\frac{(1+x)^{2n-1}}{x}\,
B_{2n-1}\!\left(-\,\dfrac{1-x}{1+x}\right),\\[1.2em]
\displaystyle
\Lambda_n(x)
=
\frac{(-1)^{n+1}}{(2^{2n+1}-1)(2n)!}\,
\frac{(1+x)^{2n-1}}{x}\,
A_{2n}\!\left(-\,\dfrac{1-x}{1+x}\right).
\end{cases}
\]
This completes the proof.
\end{proof}

\section{Recurrence formulae for the polynomials}

\begin{proposition}
For \(n\ge 1\), the polynomials \(\Xi_n\) and \(\Lambda_n\) satisfy the recurrences
\[
\Xi_1(x)\equiv \frac14,
\qquad
\Lambda_1(x)\equiv \frac17,
\]
\[
\Xi_{n+1}(x)
=
-\frac{1}{8n(2n+1)}
\left[
(x^2-1)^2\,\Xi_n''(x)
+\frac{2(x^2-1)(3x^2-1)}{x}\,\Xi_n'(x)
+(6x^2-5)\,\Xi_n(x)
\right],
\]
\[
\Lambda_{n+1}(x)
=
-\frac{2^{2n+1}-1}{(2^{2n+3}-1)(2n+1)(2n+2)}
\left[
(x^2-1)^2\,\Lambda_n''(x)
+\frac{2(x^2-1)(4x^2-1)}{x}\,\Lambda_n'(x)
+4(3x^2-2)\,\Lambda_n(x)
\right].
\]
\end{proposition}

\begin{proof}
For an arbitrary function \(f\), define the operators \(D\) and \(L_n\) by
\[
D[f](x):=\frac{d}{dx}f(x),
\qquad
L_n[f](x):=(1+x)(1-x)^{2n-2}\,f\!\left(\frac{1+x}{1-x}\right).
\]
From \cref{prop:eulerian_polynomial_form}, we immediately deduce that
\begin{equation}\label{eq:AB_by_pol}
B_{2n-1}(x)
=
(-1)^{n+1}\,2^{2n}(2n-1)!\,L_n[\Xi_n](x),\quad
A_{2n}(x)
=
(-1)^{n+1}\,\frac{(2^{2n+1}-1)(2n)!}{2^{2n-1}}\,L_n[\Lambda_n](x).
\end{equation}

Differentiating once and twice with respect to \(x\), we obtain
\begin{equation}\label{eq:AB_by_pol_derivatives_fully_expanded}
\begin{aligned}
B_{2n-1}'(x)
&= (-1)^{n+1}2^{2n}(2n-1)!\,D\circ L_n[\Xi_n](x),
&
B_{2n-1}''(x)
&= (-1)^{n+1}2^{2n}(2n-1)!\,D^2\circ L_n[\Xi_n](x),
\\
A_{2n}'(x)
&= (-1)^{n+1}\dfrac{(2^{2n+1}-1)(2n)!}{2^{2n-1}}\,D\circ L_n[\Lambda_n](x),
&
A_{2n}''(x)
&= (-1)^{n+1}\dfrac{(2^{2n+1}-1)(2n)!}{2^{2n-1}}\,D^2\circ L_n[\Lambda_n](x).
\end{aligned}
\end{equation}

A direct computation yields
\[
D\circ L_n[f](x)
=
(1-x)^{2n-3}\bigl(3-2n-(2n-1)x\bigr)\,
f\!\left(\frac{1+x}{1-x}\right)
+2(1+x)(1-x)^{2n-4}\,
f'\!\left(\frac{1+x}{1-x}\right),
\]
\begin{align*}
D^2\circ L_n[f](x)
=
&\,2(n-1)\bigl((2n-1)x+2n-5\bigr)(1-x)^{2n-4}
f\!\left(\frac{1+x}{1-x}\right)
-\,8(nx+n-x-2)(1-x)^{2n-5}
f'\!\left(\frac{1+x}{1-x}\right)
\\
&+\,4(1+x)(1-x)^{2n-6}
f''\!\left(\frac{1+x}{1-x}\right).
\end{align*}

Again by \cref{prop:eulerian_polynomial_form}, after replacing \(n\) with \(n+1\), we obtain
\begin{equation}\label{eq:pol_by_AB_shifted}
\begin{cases}
\displaystyle
\Xi_{n+1}\!\left(\dfrac{1+x}{1-x}\right)
=
\dfrac{(-1)^n}{2^{2n+2}(2n+1)!\,(1+x)(1-x)^{2n}}\,
B_{2n+1}(x),\\[1.2em]
\displaystyle
\Lambda_{n+1}\!\left(\dfrac{1+x}{1-x}\right)
=
\dfrac{(-1)^n\,2^{2n+1}}{(2^{2n+3}-1)(2n+2)!\,(1+x)(1-x)^{2n}}\,
A_{2n+2}(x).
\end{cases}
\end{equation}

Applying the recurrence relations for the Eulerian polynomials twice, we may rewrite \(B_{2n+1}\) and \(A_{2n+2}\) solely in terms of \(B_{2n-1},B_{2n-1}',B_{2n-1}''\) and \(A_{2n},A_{2n}',A_{2n}''\). This gives
\begin{equation*}
\begin{cases}
\displaystyle
\Xi_{n+1}\!\left(\dfrac{1+x}{1-x}\right)
=
\dfrac{(-1)^n}{2^{2n+2}(2n+1)!\,(1+x)(1-x)^{2n}}
\Bigl[
\bigl(16n^2x^2-8nx^2+16nx+x^2-2x+1\bigr)B_{2n-1}(x)
\\[0.6em]\hspace{8em}
+\,8x(1-x)\bigl((2n-1)x+1\bigr)B_{2n-1}'(x)
+\,4x^2(1-x)^2B_{2n-1}''(x)
\Bigr],\\[1.6em]
\displaystyle
\Lambda_{n+1}\!\left(\dfrac{1+x}{1-x}\right)
=
\dfrac{(-1)^n\,2^{2n+1}}{(2^{2n+3}-1)(2n+2)!\,(1+x)(1-x)^{2n}}
\Bigl[
\bigl(4n^2x^2+(6n+1)x+1\bigr)A_{2n}(x)
\\[0.6em]\hspace{8em}
+\,x(1-x)\bigl((4n-1)x+3\bigr)A_{2n}'(x)
+\,x^2(1-x)^2A_{2n}''(x)
\Bigr].
\end{cases}
\end{equation*}

Substituting \eqref{eq:AB_by_pol} and \eqref{eq:AB_by_pol_derivatives_fully_expanded} into these identities and simplifying, we obtain
\begin{equation*}
\begin{cases}
\displaystyle
\Xi_{n+1}\!\left(\dfrac{1+x}{1-x}\right)
=
-\dfrac{x^2+22x+1}{8n(2n+1)(1-x)^2}\,
\Xi_n\!\left(\dfrac{1+x}{1-x}\right)
\\[1.2em]\hspace{3.2em}
-\dfrac{2x(x^2+4x+1)}{n(2n+1)(1+x)(1-x)^3}\,
\Xi_n'\!\left(\dfrac{1+x}{1-x}\right)
-\dfrac{2x^2}{n(2n+1)(1-x)^4}\,
\Xi_n''\!\left(\dfrac{1+x}{1-x}\right),
\\[1.8em]
\displaystyle
\Lambda_{n+1}\!\left(\dfrac{1+x}{1-x}\right)
=
-\dfrac{2(2^{2n+1}-1)(x^2+10x+1)}{(2^{2n+3}-1)(n+1)(2n+1)(1-x)^2}\,
\Lambda_n\!\left(\dfrac{1+x}{1-x}\right)
\\[1.2em]\hspace{3.2em}
-\dfrac{4(2^{2n+1}-1)x(x+3)(3x+1)}{(2^{2n+3}-1)(n+1)(2n+1)(1+x)(1-x)^3}\,
\Lambda_n'\!\left(\dfrac{1+x}{1-x}\right)
\\[1.2em]\hspace{3.2em}
-\dfrac{8(2^{2n+1}-1)x^2}{(2^{2n+3}-1)(n+1)(2n+1)(1-x)^4}\,
\Lambda_n''\!\left(\dfrac{1+x}{1-x}\right).
\end{cases}
\end{equation*}

Finally, apply the inverse Möbius transformation
\(
t=\dfrac{x-1}{x+1}.
\)
This yields
\begin{equation*}
\begin{cases}
\displaystyle
\Xi_{n+1}(t)
=
-\frac{1}{8n(2n+1)}
\left[
(t^2-1)^2\,\Xi_n''(t)
+\frac{2(t^2-1)(3t^2-1)}{t}\,\Xi_n'(t)
+(6t^2-5)\,\Xi_n(t)
\right],
\\[1.8em]
\displaystyle
\Lambda_{n+1}(t)
=
-\frac{2^{2n+1}-1}{(2^{2n+3}-1)(2n+1)(2n+2)}
\left[
(t^2-1)^2\,\Lambda_n''(t)
+\frac{2(t^2-1)(4t^2-1)}{t}\,\Lambda_n'(t)
+4(3t^2-2)\,\Lambda_n(t)
\right].
\end{cases}
\end{equation*}
Renaming \(t\) as \(x\), we obtain the desired recurrences.

To conclude, note that \(\Xi_n\) and \(\Lambda_n\) are even polynomials for all \(n\). Hence \(\Xi_n'\) and \(\Lambda_n'\) are odd polynomials, so
\[
\frac{\Xi_n'(x)}{x}
\qquad\text{and}\qquad
\frac{\Lambda_n'(x)}{x}
\]
are again polynomials. Therefore the right-hand sides above are indeed polynomials, as claimed.

This completes the proof.
\end{proof}

\vspace{0.4cm}

\part{Algebraic and analytic structure of the polynomials}\label[part]{part:algebraic_and_analytic_structure}

\section{Second coefficient}

\begin{proposition}\label{prop:Cn-n-2}
Let $\mathcal C^{(A)}_{n,t}, \mathcal C^{(B)}_{n,t}$ be defined as in \cref{theorem:explicit_polynomials}. For \(n\ge2\), one has
\[
\mathcal C^{(A)}_{n,n-2}
=
-\frac{n-1}{3}\,(2n)!,
\qquad
\mathcal C^{(B)}_{n,n-2}
=
-\frac{4n-3}{3}\,2^{2n-3}(2n-1)!.
\]
\end{proposition}

\begin{proof}
Fix \(k\in\{0,\dots,n-1\}\) and set \(M:=2n-2k-1\). By \cref{lemma:Enk},
\[
E_{n,k}(x)
=
2x\sum_{t=0}^{n-1}
\left(
\sum_{i=0}^{k}(-1)^i\binom{k}{i}\binom{M}{2t-2i+1}
\right)x^{2t}.
\]
Hence, for \(t=n-2\), the coefficient of \(x^{2n-3}\) in \(E_{n,k}(x)\) equals
\[
2\sum_{i=0}^{k}(-1)^i\binom{k}{i}\binom{M}{2n-3-2i}.
\]

On the other hand,
\[
E_{n,k}(x)=(1-x^2)^k\bigl((1+x)^M-(1-x)^M\bigr).
\]
Since \(M\) is odd, we have
\[
(1+x)^M-(1-x)^M
=
2x^M+2\binom{M}{2}x^{M-2}+\text{terms of degree }\le M-4,
\]
whereas
\[
(1-x^2)^k
=
(-1)^k x^{2k}+(-1)^{k-1}k\,x^{2k-2}
+\text{terms of degree }\le 2k-4.
\]
Because \(\deg E_{n,k}=2k+M=2n-1\), only these two displayed terms contribute to
the coefficient of \(x^{2n-3}=x^{2k+M-2}\). Therefore
\[
[x^{2n-3}]\,E_{n,k}(x)
=
2(-1)^k\binom{M}{2}+2(-1)^{k-1}k
=
2(-1)^k\left(\binom{M}{2}-k\right).
\]
Comparing coefficients yields
\[
\sum_{i=0}^{k}(-1)^i\binom{k}{i}\binom{2n-2k-1}{2n-3-2i}
=
(-1)^k\left(\binom{2n-2k-1}{2}-k\right).
\]

Substituting this into the definitions of \(\mathcal C^{(A)}_{n,t}\) and
\(\mathcal C^{(B)}_{n,t}\) with \(t=n-2\), the factors \((-1)^k\) cancel, and we obtain
\[
\mathcal C^{(A)}_{n,n-2}
=
\sum_{k=0}^{n-1}\left\langle{2n\atop k}\right\rangle
\left(\binom{2n-2k-1}{2}-k\right),
\]
\[
\mathcal C^{(B)}_{n,n-2}
=
\sum_{k=0}^{n-1}\left\langle{2n-1\atop k}\right\rangle^{B}
\left(\binom{2n-2k-1}{2}-k\right).
\]

Now define
\[
q_n(k):=\binom{2n-2k-1}{2}-k
=
2k^2-(4n-2)k+(2n^2-3n+1).
\]
A direct computation shows that \(q_n(2n-1-k)=q_n(k)\). Using the palindromicity
of the Eulerian numbers of type~\(A\) and type~\(B\), it follows that
\[
\mathcal C^{(A)}_{n,n-2}
=
\frac12\sum_{k=0}^{2n-1}\left\langle{2n\atop k}\right\rangle q_n(k),
\qquad
\mathcal C^{(B)}_{n,n-2}
=
\frac12\sum_{k=0}^{2n-1}\left\langle{2n-1\atop k}\right\rangle^{B} q_n(k).
\]

Using the mean and variance formulae for the Euler--Frobenius distribution from
\cite[Theorem~5.3]{Janson2013EulerFrobenius}, specialized to \(\rho=1\) and
\(\rho=\tfrac12\), we have
\[
\sum_{k=0}^{2n-1}\left\langle{2n\atop k}\right\rangle=(2n)!,
\qquad
\sum_{k=0}^{2n-1}k\left\langle{2n\atop k}\right\rangle=\frac{2n-1}{2}(2n)!,
\]
\[
\sum_{k=0}^{2n-1}k^2\left\langle{2n\atop k}\right\rangle
=
\left(n^2-\frac56n+\frac13\right)(2n)!,
\]
\[
\sum_{k=0}^{2n-1}\left\langle{2n-1\atop k}\right\rangle^{B}
=
2^{2n-1}(2n-1)!,
\qquad
\sum_{k=0}^{2n-1}k\left\langle{2n-1\atop k}\right\rangle^{B}
=
\frac{2n-1}{2}\,2^{2n-1}(2n-1)!,
\]
\[
\sum_{k=0}^{2n-1}k^2\left\langle{2n-1\atop k}\right\rangle^{B}
=
\left(n^2-\frac56n+\frac14\right)2^{2n-1}(2n-1)!.
\]

Inserting these identities into the above sums gives
\[
\mathcal C^{(A)}_{n,n-2}
=
-\frac{n-1}{3}\,(2n)!,
\qquad
\mathcal C^{(B)}_{n,n-2}
=
-\frac{4n-3}{3}\,2^{2n-3}(2n-1)!,
\]
as claimed.
\end{proof}

\section{Leading coefficients and exact degree}

\begin{proposition}\label[proposition]{prop:leading_coefficients}
Let $n\in\mathbb{N}$ and define $\mathcal{C}^{(B)}_{n,t}$ and $\mathcal{C}^{(A)}_{n,t}$ as in \cref{theorem:explicit_polynomials}. The leading coefficients of the even polynomials $\Xi_n,\Lambda_n\in\mathbb{Q}[x]$ satisfy
\[
[x^{2n-2}]\,\Xi_n(x)=\frac{(-1)^{n+1}}{2^{2n}},
\qquad
[x^{2n-2}]\,\Lambda_n(x)=\frac{(-1)^{n+1}}{2^{2n+1}-1}.
\]
Equivalently,
\[
\mathcal{C}^{(B)}_{n,n-1}=2^{2n-2}(2n-1)!,
\qquad
\mathcal{C}^{(A)}_{n,n-1}=\frac{(2n)!}{2}.
\]
\end{proposition}

\begin{proof}
Fix $n\in\mathbb{N}$ and $k\in\mathbb{N}_0$ with $0\le k\le n-1$.
Recall from \cref{lemma:Enk} that
\[
E_{n,k}(x)=(1-x^2)^k\Bigl((1+x)^{2n-2k-1}-(1-x)^{2n-2k-1}\Bigr),
\]
and that $E_{n,k}$ is an odd polynomial of degree $2n-1$ admitting the representation
\[
E_{n,k}(x)
=
2x\sum_{t=0}^{n-1}\left(\sum_{i=0}^{k}(-1)^i\binom{k}{i}\binom{2n-2k-1}{2t-2i+1}\right)x^{2t}.
\]

On the one hand, the leading term of $(1-x^2)^k$ is $(-1)^k x^{2k}$, while the leading term of
$(1+x)^{2n-2k-1}-(1-x)^{2n-2k-1}$ is $2x^{2n-2k-1}$. Hence
\[
[x^{2n-1}]\,E_{n,k}(x)=2(-1)^k.
\]

On the other hand, in the above series representation the monomial $x^{2n-1}$ arises only from the index
$t=n-1$, so that
\[
[x^{2n-1}]\,E_{n,k}(x)
=
2\sum_{i=0}^{k}(-1)^i\binom{k}{i}\binom{2n-2k-1}{2n-2i-1}.
\]
Comparing the two expressions yields the identity
\begin{equation}\label{eq:inner_sum_top_degree}
\sum_{i=0}^{k}(-1)^i\binom{k}{i}\binom{2n-2k-1}{2n-2i-1}=(-1)^k.
\end{equation}

Now set $t=n-1$ in the definitions of $\mathcal{C}^{(B)}_{n,t}$ and $\mathcal{C}^{(A)}_{n,t}$ from
\cref{theorem:explicit_polynomials} and apply \eqref{eq:inner_sum_top_degree}. This gives
\[
\mathcal{C}^{(B)}_{n,n-1}
=
\sum_{k=0}^{n-1}\left\langle\begin{matrix}2n-1\\ k\end{matrix}\right\rangle^{B}
(-1)^k\cdot(-1)^k
=
\sum_{k=0}^{n-1}\left\langle\begin{matrix}2n-1\\ k\end{matrix}\right\rangle^{B},
\]
and similarly
\[
\mathcal{C}^{(A)}_{n,n-1}
=
\sum_{k=0}^{n-1}\left\langle\begin{matrix}2n\\ k\end{matrix}\right\rangle.
\]

Using the palindromicity (symmetry) of Eulerian numbers (type $A$ and type $B$) \cite{FulmanLeeKimPetersen2021},
\[
\left\langle\begin{matrix}m\\ k\end{matrix}\right\rangle
=
\left\langle\begin{matrix}m\\ m-1-k\end{matrix}\right\rangle,
\qquad
\left\langle\begin{matrix}m\\ k\end{matrix}\right\rangle^{B}
=
\left\langle\begin{matrix}m\\ m-k\end{matrix}\right\rangle^{B},
\]
the partial sums up to the midpoint equal one half of the total sums. Since
\[
\sum_{k=0}^{n-1}\left\langle\begin{matrix}n\\ k\end{matrix}\right\rangle=n!,
\qquad
\sum_{k=0}^{n}\left\langle\begin{matrix}n\\ k\end{matrix}\right\rangle^{B}=2^{n}n!,
\]
we obtain
\[
\mathcal{C}^{(A)}_{n,n-1}=\frac{(2n)!}{2},
\qquad
\mathcal{C}^{(B)}_{n,n-1}=2^{2n-2}(2n-1)!.
\]

Finally, taking the coefficient of $x^{2n-2}$ in the explicit expressions for $\Xi_n$ and $\Lambda_n$
from \cref{theorem:explicit_polynomials} yields
\[
[x^{2n-2}]\,\Xi_n(x)=\frac{(-1)^{n+1}}{2^{4n-2}(2n-1)!}\,\mathcal{C}^{(B)}_{n,n-1}
=\frac{(-1)^{n+1}}{2^{2n}},
\]
\[
[x^{2n-2}]\,\Lambda_n(x)=\frac{2(-1)^{n+1}}{(2^{2n+1}-1)(2n)!}\,\mathcal{C}^{(A)}_{n,n-1}
=\frac{(-1)^{n+1}}{2^{2n+1}-1},
\]
as claimed.
\end{proof}

\begin{corollary}\label[corollary]{cor:exact_degree}
For every $n\in\mathbb{N}$, the polynomials $\Xi_n$ and $\Lambda_n$ from \cref{theorem:explicit_polynomials}
have degree exactly $2n-2$.
\end{corollary}

\begin{proof}
By \cref{theorem:explicit_polynomials}, both $\Xi_n$ and $\Lambda_n$ are even polynomials of degree at most $2n-2$.
Moreover, by \cref{prop:leading_coefficients} their leading coefficients are
\[
[x^{2n-2}]\,\Xi_n(x)=\frac{(-1)^{n+1}}{2^{2n}}\neq 0,
\qquad
[x^{2n-2}]\,\Lambda_n(x)=\frac{(-1)^{n+1}}{2^{2n+1}-1}\neq 0.
\]
Hence neither polynomial can have degree smaller than $2n-2$, and therefore
$\deg(\Xi_n)=\deg(\Lambda_n)=2n-2$.
\end{proof}

\section{Particular values}

\begin{proposition}[Values at $x=0$]\label{prop:XiLambda-at-zero}
Let $n \in \mathbb{N}$. Then
\begin{equation}
\Xi_n(0)=\frac{(-1)^{n}}{2^{2n}(2n-1)!}\,E_{2n},
\qquad
\Lambda_n(0)=\frac{(-1)^{n}\,2^{2n+1}\bigl(2^{2n+2}-1\bigr)}{(2^{2n+1}-1)(n+1)(2n)!}\,B_{2n+2},
\label{eq:values_at_zero}
\end{equation}

where $E_{2n}$ are the Euler (secant) numbers and $B_{2n+2}$ are the Bernoulli numbers.
\end{proposition}

\begin{proof}
By \cref{prop:generating_functions_Xi_Lambdahat}, we have
\[
\sum_{n\ge1}\Xi_n(x)\,u^{2n-1}
=
\frac{\sin(u/2)}{1+x^2-(x^2-1)\cos u},
\qquad
\sum_{n\ge1}\bigl(2^{2n+1}-1\bigr)\Lambda_n(x)\,u^{2n}
=
\frac{1-\cos(2u)}{1+x^2-(x^2-1)\cos(2u)}.
\]

Setting $\displaystyle x=0$, the first identity becomes $\displaystyle \sum_{n\ge1}\Xi_n(0)\,u^{2n-1}=\frac{\sin(u/2)}{1+\cos u}$. Since $\displaystyle 1+\cos u=2\cos^2(u/2)$, this simplifies to $\displaystyle \sum_{n\ge1}\Xi_n(0)\,u^{2n-1}=\frac12\tan(u/2)\sec(u/2)$. Using $\displaystyle \frac{d}{du}\sec(u/2)=\frac12\tan(u/2)\sec(u/2)$, we obtain
\[
\sum_{n\ge1}\Xi_n(0)\,u^{2n-1}=\frac{d}{du}\sec(u/2).
\]
Now recall the classical expansion $\displaystyle \sec z=\sum_{m\ge0}\frac{(-1)^mE_{2m}}{(2m)!}\,z^{2m}$. Substituting $\displaystyle z=u/2$ and differentiating termwise gives
\[
\frac{d}{du}\sec(u/2)
=
\sum_{m\ge1}\frac{(-1)^mE_{2m}}{2^{2m}(2m-1)!}\,u^{2m-1}.
\]
Comparing coefficients of $\displaystyle u^{2n-1}$ yields $\displaystyle \Xi_n(0)=\frac{(-1)^n}{2^{2n}(2n-1)!}\,E_{2n}$.

Next, setting $\displaystyle x=0$ in the second generating function gives
\[
\sum_{n\ge1}\bigl(2^{2n+1}-1\bigr)\Lambda_n(0)\,u^{2n}
=
\frac{1-\cos(2u)}{1+\cos(2u)}
=
\tan^2u.
\]
We now use the standard Bernoulli expansion $\displaystyle \tan u=\sum_{m\ge1}\frac{(-1)^{m-1}2^{2m}\bigl(2^{2m}-1\bigr)}{(2m)!}\,B_{2m}\,u^{2m-1}$. Differentiating termwise and using $\displaystyle \frac{d}{du}\tan u=\sec^2u=1+\tan^2u$, we find
\[
1+\tan^2u
=
\sum_{m\ge1}\frac{(-1)^{m-1}2^{2m}\bigl(2^{2m}-1\bigr)(2m-1)}{(2m)!}\,B_{2m}\,u^{2m-2}.
\]
The term with $\displaystyle m=1$ equals $\displaystyle 1$, hence
\[
\tan^2u
=
\sum_{m\ge2}\frac{(-1)^{m-1}2^{2m}\bigl(2^{2m}-1\bigr)(2m-1)}{(2m)!}\,B_{2m}\,u^{2m-2}.
\]
Replacing $\displaystyle m$ by $\displaystyle n+1$, we obtain
\[
\tan^2u
=
\sum_{n\ge1}\frac{(-1)^n2^{2n+2}\bigl(2^{2n+2}-1\bigr)(2n+1)}{(2n+2)!}\,B_{2n+2}\,u^{2n}.
\]
Since $\displaystyle \frac{2n+1}{(2n+2)!}=\frac{1}{(2n+2)(2n)!}$, this becomes
\begin{equation}
\tan^2u
=
\sum_{n\ge1}\frac{(-1)^n2^{2n+2}\bigl(2^{2n+2}-1\bigr)}{(2n+2)(2n)!}\,B_{2n+2}\,u^{2n}.
\label{eq:series_of_tan_square}
\end{equation}
Comparing coefficients of $\displaystyle u^{2n}$ gives $\displaystyle \bigl(2^{2n+1}-1\bigr)\Lambda_n(0)=\frac{(-1)^n2^{2n+2}\bigl(2^{2n+2}-1\bigr)}{(2n+2)(2n)!}\,B_{2n+2}$, and therefore
\[
\Lambda_n(0)=
\frac{(-1)^n\,2^{2n+2}\bigl(2^{2n+2}-1\bigr)}{(2^{2n+1}-1)(2n+2)(2n)!}\,B_{2n+2}.
\]
This proves the claim.
\end{proof}

\begin{corollary}
Let \(n\in \mathbb{N}\).
\[
\begin{cases}
\displaystyle\sum_{k=0}^{n-1}\Big\langle {2n-1 \atop k}\Big\rangle^{B}(-1)^k(2n-2k-1) = -2^{\,2n-2}\,E_{2n},
\\
\hspace{0.1cm}\\
\displaystyle\sum_{k=0}^{n-1}\Big\langle {2n \atop k}\Big\rangle(-1)^k(2n-2k-1) = -\,\dfrac{2^{\,2n+1}\bigl(2^{2n+2}-1\bigr)}{2n+2}\,B_{2n+2}.
\end{cases}\]
\end{corollary}

\begin{proof}
Evaluating \eqref{eq:XiLambda-explicit} at $x=0$ leaves only the term $t=0$, hence
\begin{equation}
\Xi_n(0)=\frac{(-1)^{n+1}}{2^{4n-2}(2n-1)!}\,\mathcal C^{(B)}_{n,0},
\qquad
\Lambda_n(0)=\frac{2(-1)^{n+1}}{(2^{2n+1}-1)(2n)!}\,\mathcal C^{(A)}_{n,0}.
\label{eq:raw_formula}
\end{equation}
By definition, for $t=0$ we have
\[
\mathcal C^{(B)}_{n,0}
=
\sum_{k=0}^{n-1}\left\langle {2n-1 \atop k}\right\rangle^{B}(-1)^{k}
\sum_{i=0}^{k}(-1)^i\binom{k}{i}\binom{2n-2k-1}{\,1-2i\,},
\]
\[
\mathcal C^{(A)}_{n,0}
=
\sum_{k=0}^{n-1}\left\langle {2n \atop k}\right\rangle(-1)^{k}
\sum_{i=0}^{k}(-1)^i\binom{k}{i}\binom{2n-2k-1}{\,1-2i\,}.
\]
Since $\displaystyle\binom{2n-2k-1}{\,1-2i\,}=0$ for all $i\ge 1$ (negative lower index), only the term $i=0$ contributes, and thus
\[
\sum_{i=0}^{k}(-1)^i\binom{k}{i}\binom{2n-2k-1}{\,1-2i\,}
=\binom{2n-2k-1}{1}=2n-2k-1.
\]
Therefore,
\[
\mathcal C^{(B)}_{n,0}
=
\sum_{k=0}^{n-1}\left\langle {2n-1 \atop k}\right\rangle^{B}(-1)^{k}\,(2n-2k-1),
\qquad
\mathcal C^{(A)}_{n,0}
=
\sum_{k=0}^{n-1}\left\langle {2n \atop k}\right\rangle(-1)^{k}\,(2n-2k-1).
\]

Replacing these equalities and \eqref{eq:values_at_zero} in \eqref{eq:raw_formula} and performing some algebra yields
\[
\begin{cases}
\displaystyle\sum_{k=0}^{n-1}\Big\langle {2n-1 \atop k}\Big\rangle^{B}(-1)^k(2n-2k-1) = -2^{\,2n-2}\,E_{2n},
\\
\hspace{0.1cm}\\
\displaystyle\sum_{k=0}^{n-1}\Big\langle {2n \atop k}\Big\rangle(-1)^k(2n-2k-1) = -\,\dfrac{2^{\,2n+1}\bigl(2^{2n+2}-1\bigr)}{2n+2}\,B_{2n+2}.
\end{cases}\]
This proves the claim.
\end{proof}

\begin{proposition}[Values at $x=1$]\label{prop:XiLambda-at-one}
Let $n \in \mathbb{N}$. Then
\[
\Xi_n(1)=\frac{(-1)^{n-1}}{2^{2n}(2n-1)!},
\qquad
\Lambda_n(1)=\frac{(-1)^{n+1}\,2^{2n-1}}{(2^{2n+1}-1)(2n)!}.
\]
\end{proposition}

\begin{proof}
By \cref{prop:generating_functions_Xi_Lambdahat}, we have
\[
\sum_{n\ge1}\Xi_n(x)\,u^{2n-1}
=
\frac{\sin(u/2)}{1+x^2-(x^2-1)\cos u},
\qquad
\sum_{n\ge1}\bigl(2^{2n+1}-1\bigr)\Lambda_n(x)\,u^{2n}
=
\frac{1-\cos(2u)}{1+x^2-(x^2-1)\cos(2u)}.
\]

Setting $\displaystyle x=1$, the first identity becomes $\displaystyle \sum_{n\ge1}\Xi_n(1)\,u^{2n-1}=\frac{\sin(u/2)}{2}$. Using the Taylor expansion $\displaystyle \sin(u/2)=\sum_{m\ge0}\frac{(-1)^m}{(2m+1)!}\left(\frac{u}{2}\right)^{2m+1}$, we obtain $\displaystyle \frac{\sin(u/2)}{2}=\sum_{m\ge0}\frac{(-1)^m}{2^{2m+2}(2m+1)!}\,u^{2m+1}$. Replacing $\displaystyle m$ by $\displaystyle n-1$, this becomes
\[
\frac{\sin(u/2)}{2}
=
\sum_{n\ge1}\frac{(-1)^{n-1}}{2^{2n}(2n-1)!}\,u^{2n-1}.
\]
Comparing coefficients of $\displaystyle u^{2n-1}$ yields $\displaystyle \Xi_n(1)=\frac{(-1)^{n-1}}{2^{2n}(2n-1)!}$.

Next, setting $\displaystyle x=1$ in the second generating function gives $\displaystyle \sum_{n\ge1}\bigl(2^{2n+1}-1\bigr)\Lambda_n(1)\,u^{2n}=\frac{1-\cos(2u)}{2}$. Using the Taylor expansion $\displaystyle \cos(2u)=\sum_{m\ge0}\frac{(-1)^m(2u)^{2m}}{(2m)!}$, we get $\displaystyle \frac{1-\cos(2u)}{2}=\sum_{m\ge1}\frac{(-1)^{m+1}2^{2m-1}}{(2m)!}\,u^{2m}$. Renaming $\displaystyle m$ as $\displaystyle n$, we obtain
\[
\frac{1-\cos(2u)}{2}
=
\sum_{n\ge1}\frac{(-1)^{n+1}2^{2n-1}}{(2n)!}\,u^{2n}.
\]
Comparing coefficients of $\displaystyle u^{2n}$ gives $\displaystyle \bigl(2^{2n+1}-1\bigr)\Lambda_n(1)=\frac{(-1)^{n+1}2^{2n-1}}{(2n)!}$, and hence
\[
\Lambda_n(1)=\frac{(-1)^{n+1}\,2^{2n-1}}{(2^{2n+1}-1)(2n)!}.
\]
This proves the claim.
\end{proof}

\section{Real-rootedness, constant sign, log-concavity, unimodality and uniform convergence}

\begin{theorem}\label{theorem:real-rootedness}
Let \(n\in\mathbb N\) with \(n\ge 2\). All zeros of \(\Lambda_n(x)\) and \(\Xi_n(x)\) lie in \((-1,1)\setminus\{0\}\). Moreover, they are real and simple, and the zeros of \(\Lambda_n(x)\) interlace those of \(\Lambda_{n+1}(x)\), while the zeros of \(\Xi_n(x)\) interlace those of \(\Xi_{n+1}(x)\).
\end{theorem}

\begin{proof}
By \cref{prop:eulerian_polynomial_form}, with \(u=\dfrac{1-x}{1+x}\), we have
\[
\Lambda_n(x)=c_n\,\dfrac{(1+x)^{2n-1}}{x}\,A_{2n}(-u),
\qquad
\Xi_n(x)=d_n\,\dfrac{(1+x)^{2n-1}}{x}\,B_{2n-1}(-u),
\]
for some nonzero constants \(c_n,d_n\in\mathbb Q\).

It is classical that the Eulerian polynomials \(A_m(t)\) have only real zeros
\cite{Frobenius1910BernoulliEuler,Janson2013EulerFrobenius,VisontaiWilliams2013StableWEulerian},
and that the type~\(B\) Eulerian polynomials \(B_m(t)\) have only real zeros as well
\cite{Brenti1994,VisontaiWilliams2013StableWEulerian}. Moreover, Chow
\cite{Chow2024InterlacingEulerianIII} proved that these zeros are simple and negative.
It follows that all zeros of \(A_{2n}(-u)\) and \(B_{2n-1}(-u)\) are real, simple, and lie in \((0,\infty)\).

Since \(A_{2n}(t)\) and \(B_{2n-1}(t)\) are palindromic of odd degree, we have
\(A_{2n}(-1)=0\) and \(B_{2n-1}(-1)=0\). Hence \(u=1\) is a simple zero of both \(A_{2n}(-u)\) and \(B_{2n-1}(-u)\), so
\[
\widetilde A_{2n}(u):=\dfrac{A_{2n}(-u)}{1-u},
\qquad
\widetilde B_{2n-1}(u):=\dfrac{B_{2n-1}(-u)}{1-u}
\]
are polynomials whose zeros are still real, simple, and positive, now all distinct from \(1\).

Using \(1-u=1-\dfrac{1-x}{1+x}=\dfrac{2x}{1+x}\), we may rewrite the above identities as
\[
\Lambda_n(x)=\tilde c_n\,(1+x)^{2n-2}\,\widetilde A_{2n}(u),
\qquad
\Xi_n(x)=\tilde d_n\,(1+x)^{2n-2}\,\widetilde B_{2n-1}(u),
\]
for suitable nonzero constants \(\tilde c_n,\tilde d_n\). By \cref{prop:XiLambda-at-one}, neither \(\Lambda_n\) nor \(\Xi_n\) vanishes at \(x=-1\), so the factor \((1+x)^{2n-2}\) contributes no zero. Thus all zeros of \(\Lambda_n\) and \(\Xi_n\) come from \(\widetilde A_{2n}(u)\) and \(\widetilde B_{2n-1}(u)\).

Finally, the map \(u\mapsto x=\dfrac{1-u}{1+u}\) is a strictly decreasing bijection from \((0,\infty)\setminus\{1\}\) onto \((-1,1)\setminus\{0\}\). Hence every zero \(u_0>0\), \(u_0\neq 1\), of \(\widetilde A_{2n}\) or \(\widetilde B_{2n-1}\) yields a unique zero \(x_0=\dfrac{1-u_0}{1+u_0}\in(-1,1)\setminus\{0\}\), and simplicity is preserved. Therefore all zeros of \(\Lambda_n\) and \(\Xi_n\) are real, simple, and lie in \((-1,1)\).

For the interlacing, Chow \cite{Chow2024InterlacingEulerianIII} proved that the zeros of \(A_{2n}(t)\) and \(A_{2n+2}(t)\) interlace, and likewise those of \(B_{2n-1}(t)\) and \(B_{2n+1}(t)\). Replacing \(t\) by \(-u\) preserves interlacing, and removing the common simple zero \(u=1\) preserves it as well. Since \(x=\dfrac{1-u}{1+u}\) is strictly monotone, interlacing is transported to the \(x\)-variable. Hence the zeros of \(\Lambda_n\) interlace those of \(\Lambda_{n+1}\), and the zeros of \(\Xi_n\) interlace those of \(\Xi_{n+1}\).
\end{proof}

\begin{corollary}\label[corollary]{cor:alternating-signs-Xi-Lambda}
Let $n\in\mathbb N$ with $n\ge 2$. The coefficients of the polynomials
$\Lambda_n(\sqrt{x})$ and $\Xi_n(\sqrt{x})$ have alternating signs.
\end{corollary}

\begin{proof}
Set
\[
\widetilde{\Lambda}_n(y):=\Lambda_n(\sqrt y),
\qquad
\widetilde{\Xi}_n(y):=\Xi_n(\sqrt y).
\]
Since $\Lambda_n$ and $\Xi_n$ are even, their zeros occur in pairs $\pm\alpha$.
By \cref{theorem:real-rootedness}, all zeros of $\Lambda_n$ and $\Xi_n$ lie in
$(-1,1)$, and by \cref{prop:XiLambda-at-zero} neither polynomial vanishes at
$0$. It follows that all zeros of $\widetilde{\Lambda}_n$ and
$\widetilde{\Xi}_n$ are real and strictly positive.

Hence we may write
\[
\widetilde{\Lambda}_n(y)=a\prod_{j=1}^{n-1}(y-\rho_j),
\qquad
\widetilde{\Xi}_n(y)=b\prod_{j=1}^{n-1}(y-\sigma_j),
\]
with $a,b\in\mathbb R\setminus\{0\}$ and $\rho_j,\sigma_j>0$. Expanding these
products, we obtain
\[
\widetilde{\Lambda}_n(y)
=
a\sum_{t=0}^{n-1}(-1)^{n-1-t}e_{n-1-t}(\rho_1,\dots,\rho_{n-1})\,y^t,
\]
\[
\widetilde{\Xi}_n(y)
=
b\sum_{t=0}^{n-1}(-1)^{n-1-t}e_{n-1-t}(\sigma_1,\dots,\sigma_{n-1})\,y^t,
\]
where $e_k$ denotes the $k$-th elementary symmetric polynomial. Since
$\rho_j,\sigma_j>0$, all these elementary symmetric polynomials are positive.
Therefore the coefficients of $\widetilde{\Lambda}_n$ and
$\widetilde{\Xi}_n$ alternate in sign.
\end{proof}

\begin{corollary}\label{cor:nozeros-outside-1}
Let $n\in\mathbb N$. Then, for every $x>1$,
\[
(-1)^{n+1}\Xi_n(x)>0,
\qquad
(-1)^{n+1}\Lambda_n(x)>0.
\]
In particular, neither $\Xi_n$ nor $\Lambda_n$ has a real zero in $(1,\infty)$.
\end{corollary}

\begin{proof}
By \cref{theorem:real-rootedness}, all zeros of $\Xi_n$ and $\Lambda_n$ lie in
$(-1,1)$. Thus both polynomials have constant sign on $(1,\infty)$. Moreover,
by \cref{prop:XiLambda-at-one},
\[
(-1)^{n+1}\Xi_n(1)>0,
\qquad
(-1)^{n+1}\Lambda_n(1)>0.
\]
Hence the same inequalities hold for every $x>1$, which proves the claim.
\end{proof}

\begin{corollary}\label{cor:coefficients-unimodal}
Let $n\ge2$ and let $\Xi_n$ and $\Lambda_n$ be the even polynomials defined in
\cref{theorem:explicit_polynomials}, i.e.
\[
\Xi_n(x)
=
\frac{(-1)^{n+1}}{2^{4n-2}(2n-1)!}
\sum_{t=0}^{n-1} \mathcal C^{(B)}_{n,t}\,x^{2t},
\qquad
\Lambda_n(x)
=
\frac{2(-1)^{n+1}}{(2^{2n+1}-1)(2n)!}
\sum_{t=0}^{n-1} \mathcal C^{(A)}_{n,t}\,x^{2t}.
\]
Then the sequences
\(\displaystyle
\bigl(|\mathcal C^{(B)}_{n,t}|\bigr)_{t=0}^{n-1},
\qquad
\bigl(|\mathcal C^{(A)}_{n,t}|\bigr)_{t=0}^{n-1}
\)
are log-concave and hence unimodal.
\end{corollary}

\begin{proof}
Define
\[
P_n(y):=\sum_{t=0}^{n-1}\mathcal C^{(B)}_{n,t}\,y^t,
\qquad
Q_n(y):=\sum_{t=0}^{n-1}\mathcal C^{(A)}_{n,t}\,y^t.
\]
Then $\Xi_n(x)$ and $P_n(x^2)$ differ only by a nonzero constant factor, and
likewise $\Lambda_n(x)$ and $Q_n(x^2)$ differ only by a nonzero constant factor.

By \cref{theorem:real-rootedness}, all zeros of $\Xi_n$ and $\Lambda_n$ are real
and lie in $(-1,1)$. Since both polynomials are even and nonzero at $x=0$ by
\cref{prop:XiLambda-at-zero}, it follows that all zeros of $P_n$ and $Q_n$ are
real and strictly positive.

Hence we may write
\[
P_n(y)=\mathcal C^{(B)}_{n,n-1}\prod_{j=1}^{n-1}(y-\rho_j),
\qquad
Q_n(y)=\mathcal C^{(A)}_{n,n-1}\prod_{j=1}^{n-1}(y-\sigma_j),
\]
with $\rho_j,\sigma_j>0$. Expanding in terms of elementary symmetric
polynomials gives
\[
|\mathcal C^{(B)}_{n,t}|
=
|\mathcal C^{(B)}_{n,n-1}|\,
e_{n-1-t}(\rho_1,\dots,\rho_{n-1}),
\]
\[
|\mathcal C^{(A)}_{n,t}|
=
|\mathcal C^{(A)}_{n,n-1}|\,
e_{n-1-t}(\sigma_1,\dots,\sigma_{n-1}).
\]
By Newton's inequalities, the sequences
\(
\bigl(e_k(\rho_1,\dots,\rho_{n-1})\bigr)_{k=0}^{n-1}
\)
and
\(
\bigl(e_k(\sigma_1,\dots,\sigma_{n-1})\bigr)_{k=0}^{n-1}
\)
are log-concave. Therefore
\(
\bigl(|\mathcal C^{(B)}_{n,t}|\bigr)_{t=0}^{n-1}
\)
and
\(
\bigl(|\mathcal C^{(A)}_{n,t}|\bigr)_{t=0}^{n-1}
\)
are log-concave as well. Since every log-concave sequence of nonnegative real
numbers is unimodal, the proof is complete.
\end{proof}

\begin{theorem}\label{theorem:uniform-to-zero}
Let $n \in \mathbb{N}$ with $n\ge2$, and let the even polynomials
$\Xi_n,\Lambda_n\in\mathbb Q[x]$ be defined as in
\cref{theorem:explicit_polynomials}. Then, as $n\to\infty$,
\[
\Xi_n \to 0 \quad\text{uniformly on }(-1,1),
\qquad
\Lambda_n \to 0 \quad\text{uniformly on }(-1,1).
\]
\end{theorem}

\begin{proof}
By \cref{theorem:real-rootedness}, all zeros of $\Xi_n$ and $\Lambda_n$ lie in
$(-1,1)$, and by \cref{prop:XiLambda-at-zero} neither polynomial vanishes at
$x=0$. Since $\Xi_n$ and $\Lambda_n$ are even, their zeros occur in symmetric
pairs $\pm\alpha$ with $0<\alpha<1$.

Define the adapted polynomials
\[
\widetilde{\Xi}_n(y):=\Xi_n(\sqrt y),
\qquad
\widetilde{\Lambda}_n(y):=\Lambda_n(\sqrt y),
\qquad 0<y<1,
\]
so that $\Xi_n(x)=\widetilde{\Xi}_n(x^2)$ and
$\Lambda_n(x)=\widetilde{\Lambda}_n(x^2)$. The zeros of
$\widetilde{\Xi}_n$ and $\widetilde{\Lambda}_n$ are the squares of the positive
zeros of $\Xi_n$ and $\Lambda_n$, hence all lie in $(0,1)$. Therefore we may
factor
\[
\widetilde{\Xi}_n(y)=\ell^{(B)}_n\prod_{j=1}^{n-1}(y-\rho^{(B)}_{n,j}),
\qquad
\widetilde{\Lambda}_n(y)=\ell^{(A)}_n\prod_{j=1}^{n-1}(y-\rho^{(A)}_{n,j}),
\]
with $0<\rho^{(B)}_{n,j},\rho^{(A)}_{n,j}<1$, where $\ell^{(B)}_n$ and
$\ell^{(A)}_n$ denote the leading coefficients of $\widetilde{\Xi}_n$ and
$\widetilde{\Lambda}_n$, respectively.

For any $y\in[0,1]$ and any $\rho\in(0,1)$ we have $|y-\rho|\le1$. Hence
\[
|\widetilde{\Xi}_n(y)|\le |\ell^{(B)}_n|,
\qquad
|\widetilde{\Lambda}_n(y)|\le |\ell^{(A)}_n|
\qquad (0\le y\le 1).
\]
Since $x^2\in(0,1)$ for every $x\in(-1,1)$, it follows that
\[
\sup_{x\in(-1,1)}|\Xi_n(x)|
=
\sup_{y\in(0,1)}|\widetilde{\Xi}_n(y)|
\le
|\ell^{(B)}_n|,
\]
and similarly
\[
\sup_{x\in(-1,1)}|\Lambda_n(x)|
=
\sup_{y\in(0,1)}|\widetilde{\Lambda}_n(y)|
\le
|\ell^{(A)}_n|.
\]
Finally, by \cref{prop:leading_coefficients},
\[
\ell^{(B)}_n=[x^{2n-2}]\Xi_n(x)=\frac{(-1)^{n+1}}{2^{2n}},
\qquad
\ell^{(A)}_n=[x^{2n-2}]\Lambda_n(x)=\frac{(-1)^{n+1}}{2^{2n+1}-1},
\]
and both tend to $0$ as $n\to\infty$. Therefore
\[
\sup_{x\in(-1,1)}|\Xi_n(x)|\to0,
\qquad
\sup_{x\in(-1,1)}|\Lambda_n(x)|\to0,
\]
which proves the claim.
\end{proof}

\section{Convergence of smallest and greatest zeros}
\begin{theorem}[Largest zeros of $\widetilde\Xi_n$ and $\widetilde\Lambda_n$]\label{theorem:extremal-zeros-Xi-Lambda}
Let $\widetilde\Xi_n(y):=\Xi_n(\sqrt{y})$ and
$\widetilde\Lambda_n(y):=\Lambda_n(\sqrt{y})$.
Then the largest zero of $\widetilde\Xi_n$ and the largest zero of $\widetilde\Lambda_n$
both converge to $1$.
\end{theorem}
\begin{proof}
We verify the hypotheses of
\cref{theorem:bounding_largest_root_a}.
By \cref{theorem:real-rootedness} and \cref{cor:nozeros-outside-1},
all zeros of $\Xi_n$ and $\Lambda_n$ are real, simple, and lie in $(-1,1)$.
Consequently, all zeros of $\widetilde\Xi_n$ and $\widetilde\Lambda_n$
are real, simple, and lie in $(0,1) \subset (-\infty,1)$. By \cref{prop:leading_coefficients,prop:XiLambda-at-one}, we have the values of $\widetilde\Xi_n(1), \widetilde\Lambda_n(1), \operatorname{lc}(\widetilde\Lambda_n)$ and $\operatorname{lc}(\widetilde\Xi)$. Applying them gives
\[
\begin{cases}
\displaystyle
\frac{\widetilde\Xi_n(1)}{\operatorname{lc}(\widetilde\Xi_n)}
=\frac{\Xi_n(1)}{[x^{2n-2}]\Xi_n(x)}
=\frac{1}{(2n-1)!},\\[1.25em]
\displaystyle
\frac{\widetilde\Lambda_n(1)}{\operatorname{lc}(\widetilde\Lambda_n)}
=\frac{\Lambda_n(1)}{[x^{2n-2}]\Lambda_n(x)}
=\frac{2^{2n-1}}{(2n)!}.
\end{cases}
\]
Bearing in mind that $\deg \widetilde\Xi_n=\deg \widetilde\Lambda_n=n-1$, one has in particular,
\[
\left|\frac{\widetilde\Xi_n(1)}{\operatorname{lc}(\widetilde\Xi_n)}\right|^{\frac{1}{n-1}}
=\bigl((2n-1)!\bigr)^{-\frac{1}{n-1}}\xrightarrow[n\to\infty]{}0,
\qquad
\left|\frac{\widetilde\Lambda_n(1)}{\operatorname{lc}(\widetilde\Lambda_n)}\right|^{\frac{1}{n-1}}
=\left(\frac{2^{2n-1}}{(2n)!}\right)^{\frac{1}{n-1}} \xrightarrow[n\to\infty]{}0,
\]
where the vanishing limits are results of \cref{lemma:factorial_root_goes_to_infty}. All assumptions of
\cref{theorem:bounding_largest_root_a}
are therefore satisfied.
It follows that the largest zeros converge to $1$.
\end{proof}

\begin{proposition}\label{prop:endpoint-ratio-limit}
For $n\ge2$, define
\[
\widetilde{\Xi}_n(y):=\Xi_n(\sqrt y),
\qquad
\widetilde{\Lambda}_n(y):=\Lambda_n(\sqrt y).
\]
Then
\[
\lim_{n\to\infty}
\left|
\frac{\widetilde{\Xi}_n(0)}{\operatorname{lc}(\widetilde{\Xi}_n)}
\right|^{1/(n-1)}
=
\frac{4}{\pi^2},
\qquad
\lim_{n\to\infty}
\left|
\frac{\widetilde{\Lambda}_n(0)}{\operatorname{lc}(\widetilde{\Lambda}_n)}
\right|^{1/(n-1)}
=
\frac{4}{\pi^2}.
\]
\end{proposition}

\begin{proof}
Since $\widetilde{\Xi}_n(y)=\Xi_n(\sqrt y)$ and
$\widetilde{\Lambda}_n(y)=\Lambda_n(\sqrt y)$, we have
\[
\widetilde{\Xi}_n(0)=\Xi_n(0),
\qquad
\widetilde{\Lambda}_n(0)=\Lambda_n(0),
\]
\[
\operatorname{lc}(\widetilde{\Xi}_n)=[x^{2n-2}]\Xi_n(x),
\qquad
\operatorname{lc}(\widetilde{\Lambda}_n)=[x^{2n-2}]\Lambda_n(x).
\]
Hence, by \cref{prop:XiLambda-at-zero,prop:leading_coefficients},
\[
\left|
\frac{\widetilde{\Xi}_n(0)}{\operatorname{lc}(\widetilde{\Xi}_n)}
\right|
=
\frac{|E_{2n}|}{(2n-1)!},
\qquad
\left|
\frac{\widetilde{\Lambda}_n(0)}{\operatorname{lc}(\widetilde{\Lambda}_n)}
\right|
=
\frac{2^{2n+1}\bigl(2^{2n+2}-1\bigr)}{(2n)! \, (2n+2)}
\,|B_{2n+2}|.
\]

Using the classical asymptotics
\[
|E_{2n}|
\sim
\frac{4^{n+1}(2n)!}{\pi^{2n+1}},
\qquad
|B_{2m}|
\sim
\frac{2(2m)!}{(2\pi)^{2m}},
\]
we obtain
\[
\frac{|E_{2n}|}{(2n-1)!}
\sim
\frac{4^{n+1}(2n)!}{\pi^{2n+1}(2n-1)!},
\]
\[
\frac{2^{2n+1}(2^{2n+2}-1)}{(2n)! \, (2n+2)}\,|B_{2n+2}|
\sim
\frac{2^{2n+1}(2^{2n+2}-1)}{(2n)! \, (2n+2)}
\cdot
\frac{2(2n+2)!}{(2\pi)^{2n+2}}.
\]
Taking $(n-1)$-st roots in both expressions yields
\[
\left(
\frac{|E_{2n}|}{(2n-1)!}
\right)^{1/(n-1)}
\longrightarrow
\frac{4}{\pi^2},
\qquad
\left(
\frac{2^{2n+1}(2^{2n+2}-1)}{(2n)! \, (2n+2)}\,|B_{2n+2}|
\right)^{1/(n-1)}
\longrightarrow
\frac{4}{\pi^2},
\]
which proves the claim.
\end{proof}

\begin{remark}
\Cref{theorem:bounding_smallest_root_a} is not applicable to the situation of \cref{prop:endpoint-ratio-limit}. Indeed, in that proposition the relevant quantity
\[
\left|
\frac{\widetilde{\Xi}_n(0)}{\operatorname{lc}(\widetilde{\Xi}_n)}
\right|^{1/(n-1)}
\qquad\text{and}\qquad
\left|
\frac{\widetilde{\Lambda}_n(0)}{\operatorname{lc}(\widetilde{\Lambda}_n)}
\right|^{1/(n-1)}
\]
does not converge to $0$, but to $\frac{4}{\pi^2}>0$. Hence the vanishing condition required in \cref{theorem:bounding_smallest_root_a} fails.

\vspace{0.2cm}

We will instead prove in \cref{part:asymptotic_zero_distribution} by different methods that the smallest zeros converge to the left endpoint 0, and more generally investigate in greater depth how the zeros are distributed asymptotically.
\end{remark}

\section{Generating functions}

\begin{proposition}[Generating functions for $\Xi_n$ and ${\Lambda}_n$]\label[proposition]{prop:generating_functions_Xi_Lambdahat}
For $n\ge1$ and for $|u|$ sufficiently small, one has
\[
\sum_{n\ge1}\Xi_n(x)\,u^{2n-1}
=
\frac{\sin(u/2)}{1+x^2-(x^2-1)\cos u},
\]
\[
\sum_{n\ge1}\left(2^{2n+1}-1\right)\Lambda_n(x)\,u^{2n}
=
\frac{1-\cos(2u)}{1+x^2-(x^2-1)\cos(2u)}.
\]
\end{proposition}

\begin{proof}
Set $\displaystyle \widehat{\Lambda}_n(x):=(2^{2n+1}-1)\Lambda_n(x)$ and $\displaystyle t:=\frac{x-1}{x+1}$. Then $\displaystyle -\frac{1-x}{1+x}=t$ and $\displaystyle 1-t=\frac{2}{1+x}$. By the defining formulae from \cref{prop:eulerian_polynomial_form},
\[
\Xi_n(x)
=
\frac{(-1)^{n+1}}{2^{4n-1}(2n-1)!}\,
\frac{(1+x)^{2n-1}}{x}\,
B_{2n-1}(t),
\]
\[
\widehat{\Lambda}_n(x)
=
\frac{(-1)^{n+1}}{(2n)!}\,
\frac{(1+x)^{2n-1}}{x}\,
A_{2n}(t).
\]

We use the standard exponential generating functions for the Eulerian polynomials of type $A$ and $B$:
\[
\sum_{m\ge0}A_m(t)\frac{z^m}{m!}
=
\frac{(1-t)e^{z(1-t)}}{1-te^{z(1-t)}},
\qquad
\sum_{m\ge0}B_m(t)\frac{z^m}{m!}
=
\frac{(1-t)e^{z(1-t)}}{1-te^{2z(1-t)}}.
\]

For the type $B$ polynomials, taking the odd part gives
\[
\sum_{n\ge1} B_{2n-1}(t)\frac{z^{2n-1}}{(2n-1)!}
=
\frac12\left(
\frac{(1-t)e^{z(1-t)}}{1-te^{2z(1-t)}}
-
\frac{(1-t)e^{-z(1-t)}}{1-te^{-2z(1-t)}}
\right).
\]
Now choose $\displaystyle z=\frac{i(1+x)u}{4}$. Since $\displaystyle 1-t=2/(1+x)$, we have $\displaystyle z(1-t)=\frac{iu}{2}$ and $\displaystyle 2z(1-t)=iu$. Moreover, $\displaystyle (-1)^{n+1}u^{2n-1}=\frac{1}{i}(iu)^{2n-1}$. Hence
\[
\sum_{n\ge1}\Xi_n(x)\,u^{2n-1}
=
\frac{1}{4ix}
\left(
\frac{(1-t)e^{iu/2}}{1-te^{iu}}
-
\frac{(1-t)e^{-iu/2}}{1-te^{-iu}}
\right).
\]
Substituting $\displaystyle t=(x-1)/(x+1)$ and simplifying yields
\[
\sum_{n\ge1}\Xi_n(x)\,u^{2n-1}
=
\frac{\sin(u/2)}{1+x^2-(x^2-1)\cos u}.
\]

For the type $A$ polynomials, taking the even part gives
\[
\sum_{n\ge0} A_{2n}(t)\frac{z^{2n}}{(2n)!}
=
\frac12\left(
\frac{(1-t)e^{z(1-t)}}{1-te^{z(1-t)}}
+
\frac{(1-t)e^{-z(1-t)}}{1-te^{-z(1-t)}}
\right).
\]
Since $\displaystyle A_0(t)=1$, subtracting the $n=0$ term gives
\[
\sum_{n\ge1} A_{2n}(t)\frac{z^{2n}}{(2n)!}
=
\frac12\left(
\frac{(1-t)e^{z(1-t)}}{1-te^{z(1-t)}}
+
\frac{(1-t)e^{-z(1-t)}}{1-te^{-z(1-t)}}
\right)-1.
\]
Now choose $\displaystyle z=i(1+x)u$. Then again $\displaystyle 1-t=2/(1+x)$, so that $\displaystyle z(1-t)=2iu$. Also, $\displaystyle (-1)^{n+1}u^{2n}=-(iu)^{2n}$. Therefore,
\[
\sum_{n\ge1}\widehat{\Lambda}_n(x)\,u^{2n}
=
-\frac{1}{x(1+x)}
\left[
\frac12\left(
\frac{(1-t)e^{2iu}}{1-te^{2iu}}
+
\frac{(1-t)e^{-2iu}}{1-te^{-2iu}}
\right)-1
\right].
\]
Substituting $\displaystyle t=(x-1)/(x+1)$ and simplifying gives
\[
\sum_{n\ge1}\widehat{\Lambda}_n(x)\,u^{2n}
=
\frac{1-\cos(2u)}{1+x^2-(x^2-1)\cos(2u)}.
\]
This proves both identities.
\end{proof}

\section{Generating functions of odd zeta and even beta values}

\begin{proposition}[Generating functions for the normalized beta and zeta values]
For \(|u|\) sufficiently small, one has
\[
\sum_{n\ge1}\frac{\beta(2n)}{\pi^{2n-1}}\,u^{2n-1}
=
\sin(u/2)\int_0^1
\frac{x\,dx}{
\sqrt{1-x^2}\,\operatorname{arctanh}x\,
\bigl(1+x^2-(x^2-1)\cos u\bigr)},
\]
\[
\sum_{n\ge1}(2^{2n+1}-1)\frac{\zeta(2n+1)}{\pi^{2n}}\,u^{2n}
=
(1-\cos 2u)\int_0^1
\frac{x\,dx}{
\operatorname{arctanh}x\,
\bigl(1+x^2-(x^2-1)\cos 2u\bigr)}.
\]
\end{proposition}

\begin{proof}
By the integral representations
\[
\frac{\beta(2n)}{\pi^{2n-1}}
=
\int_0^1
\frac{x\,\Xi_n(x)}{\sqrt{1-x^2}\,\operatorname{arctanh}x}\,dx,
\qquad
\frac{\zeta(2n+1)}{\pi^{2n}}
=
\int_0^1
\frac{x\,\Lambda_n(x)}{\operatorname{arctanh}x}\,dx,
\]
we obtain, for \(|u|\) sufficiently small,
\[
\sum_{n\ge1}\frac{\beta(2n)}{\pi^{2n-1}}\,u^{2n-1}
=
\int_0^1
\frac{x}{\sqrt{1-x^2}\,\operatorname{arctanh}x}
\sum_{n\ge1}\Xi_n(x)\,u^{2n-1}\,dx,
\]
\[
\sum_{n\ge1}(2^{2n+1}-1)\frac{\zeta(2n+1)}{\pi^{2n}}\,u^{2n}
=
\int_0^1
\frac{x}{\operatorname{arctanh}x}
\sum_{n\ge1}(2^{2n+1}-1)\Lambda_n(x)\,u^{2n}\,dx.
\]
Now apply the generating functions of \cref{prop:generating_functions_Xi_Lambdahat}
\[
\sum_{n\ge1}\Xi_n(x)\,u^{2n-1}
=
\frac{\sin(u/2)}{1+x^2-(x^2-1)\cos u},
\]
\[
\sum_{n\ge1}(2^{2n+1}-1)\Lambda_n(x)\,u^{2n}
=
\frac{1-\cos 2u}{1+x^2-(x^2-1)\cos 2u}.
\]
Substituting these identities into the two previous formulae gives
\[
\sum_{n\ge1}\frac{\beta(2n)}{\pi^{2n-1}}\,u^{2n-1}
=
\sin(u/2)\int_0^1
\frac{x\,dx}{
\sqrt{1-x^2}\,\operatorname{arctanh}x\,
\bigl(1+x^2-(x^2-1)\cos u\bigr)},
\]
\[
\sum_{n\ge1}(2^{2n+1}-1)\frac{\zeta(2n+1)}{\pi^{2n}}\,u^{2n}
=
(1-\cos 2u)\int_0^1
\frac{x\,dx}{
\operatorname{arctanh}x\,
\bigl(1+x^2-(x^2-1)\cos 2u\bigr)}.
\]
This proves the claim.
\end{proof}

\begin{corollary}
For \(|u|\) sufficiently small, one has
\[
\sum_{n\ge1}\beta(2n)\,u^{2n-1}
=
\sin\!\left(\frac{\pi u}{2}\right)\int_0^1
\frac{x\,dx}{
\sqrt{1-x^2}\,\operatorname{arctanh}x\,
\bigl(1+x^2-(x^2-1)\cos (\pi u)\bigr)},
\]
\[
\sum_{n\ge1}(2^{2n+1}-1)\zeta(2n+1)\,u^{2n}
=
(1-\cos 2\pi u)\int_0^1
\frac{x\,dx}{
\operatorname{arctanh}x\,
\bigl(1+x^2-(x^2-1)\cos 2\pi u\bigr)}.
\]
\end{corollary}

\begin{proof}
This follows immediately from the preceding proposition by replacing \(u\) with \(\pi u\).
\end{proof}

While generating functions for special zeta-value sequences are classical, we have not found these specific integral representations for the renormalized odd zeta values and even beta values in the literature.

\section{Some particular weighted integrals involving the polynomials}

\begin{proposition}\label[proposition]{prop:explicit_formulae_selected_weights}
For every integer \(n\ge 1\), the following identities hold:
\[
\newcommand{\Tstrut}{\rule[-1.15em]{0pt}{3.4em}}
\renewcommand{\arraystretch}{1.15}
\begin{array}{|c|c|c|}
\hline
\Tstrut
w(x)
&
\displaystyle \int_0^1 w(x)\,\Xi_n(x)\,dx
&
\displaystyle \int_0^1 w(x)\,\Lambda_n(x)\,dx
\\ \hline
\Tstrut
1
&
\displaystyle \frac{(-1)^{n-1}E_{2n-2}}{2^{2n}(2n-2)!}
&
\displaystyle \frac{2^{2n}(2^{2n}-1)(-1)^{n+1}B_{2n}}{(2^{2n+1}-1)(2n)!}
\\ \hline
\Tstrut
x^2
&
\displaystyle \frac{(-1)^{n+1}(2^{2n-1}-1)B_{2n}}{2^{2n-1}(2n-1)!}
&
\displaystyle \frac{(-1)^{n-1}2^{2n}B_{2n}}{(2^{2n+1}-1)(2n)!}
\\ \hline
\Tstrut
\sqrt{1-x^2}
&
\displaystyle \frac{\pi(2^{2n-1}-1)(2^{2n}-1)(-1)^{n+1}B_{2n}}{2^{2n}(2n)!}
&
\displaystyle \frac{\pi(-1)^nE_{2n}}{2(2^{2n+1}-1)(2n)!}
\\ \hline
\Tstrut
\displaystyle \frac{1}{\sqrt{1-x^2}}
&
\displaystyle \pi\,\frac{2^{2n}-1}{2}\,\frac{(-1)^{n+1}B_{2n}}{(2n)!}
&
\displaystyle \frac{\pi(-1)^n(E_{2n}-1)}{2(2^{2n+1}-1)(2n)!}
\\ \hline
\Tstrut
x^2\sqrt{1-x^2}
&
\displaystyle \frac{\pi n(2^{2n+2}-1)(-1)^nB_{2n+2}}{2^{2n+1}(n+1)(2n)!}
&
\displaystyle \frac{\pi(2^{2n+2}-1)(-1)^nB_{2n+2}}{2(2^{2n+1}-1)(n+1)(2n)!}
\\ \hline
\Tstrut
\displaystyle \frac{x^2}{\sqrt{1-x^2}}
&
\displaystyle \frac{\pi(2^{2n}-1)(-1)^{n+1}B_{2n}}{2^{2n}(2n)!}
&
\displaystyle \frac{\pi(-1)^{n+1}}{2(2^{2n+1}-1)(2n)!}
\\ \hline
\end{array}
\]
\end{proposition}

\begin{proof}
For a fixed weight \(w\), set
\[
I_n^{(B)}(w):=\int_0^1 w(x)\,\Xi_n(x)\,dx,
\qquad
I_n^{(A)}(w):=\int_0^1 w(x)\,\Lambda_n(x)\,dx.
\]
Using \cref{prop:generating_functions_Xi_Lambdahat}, for \(|u|\) sufficiently small we obtain
\[
\sum_{n\ge1} I_n^{(B)}(w)\,u^{2n-1}
=
\sin(u/2)\,K_w(u),
\]
\[
\sum_{n\ge1}(2^{2n+1}-1)I_n^{(A)}(w)\,u^{2n}
=
(1-\cos 2u)\,K_w(2u),
\]
where
\[
K_w(\phi):=
\int_0^1 \frac{w(x)\,dx}{1+x^2-(x^2-1)\cos\phi}
=
\int_0^1 \frac{w(x)\,dx}{(1+\cos\phi)+(1-\cos\phi)x^2}.
\]

Thus everything reduces to evaluating \(K_w(\phi)\). For the six weights in the statement one finds, by elementary substitutions and partial fraction manipulations, that
\[
\renewcommand{\arraystretch}{1.5}
\begin{array}{|c|c|}
\hline
w(x) & K_w(\phi) \\ \hline
1
&
\displaystyle \frac{\phi}{2\sin\phi}
\\ \hline
x^2
&
\displaystyle
\frac{1}{1-\cos\phi}
-
\frac{1+\cos\phi}{1-\cos\phi}\,\frac{\phi}{2\sin\phi}
\\ \hline
\sqrt{1-x^2}
&
\displaystyle
\frac{\pi}{2(1-\cos\phi)}
\left(\frac{2}{\sqrt{2(1+\cos\phi)}}-1\right)
\\ \hline
\displaystyle \frac{1}{\sqrt{1-x^2}}
&
\displaystyle \frac{\pi}{4\cos(\phi/2)}
\\ \hline
x^2\sqrt{1-x^2}
&
\displaystyle \frac{\pi}{32}\sec^4(\phi/4)
\\ \hline
\displaystyle \frac{x^2}{\sqrt{1-x^2}}
&
\displaystyle \frac{\pi}{8}\sec^2(\phi/4)
\\ \hline
\end{array}
\]
For example, the cases \(w(x)=1/\sqrt{1-x^2}\) and \(w(x)=x^2/\sqrt{1-x^2}\) follow from the substitution \(x=\sin\theta\), while the cases \(w(x)=x^2\) and \(w(x)=x^2\sqrt{1-x^2}\) are obtained from the corresponding \(w(x)=1\) and \(w(x)=\sqrt{1-x^2}\) cases by the identity
\[
\frac{x^2}{a+bx^2}=\frac1b-\frac{a}{b(a+bx^2)}.
\]

Substituting these evaluations into the generating series gives
\[
\renewcommand{\arraystretch}{1.5}
\begin{array}{|c|c|c|}
\hline
w(x)
&
\displaystyle \sum_{n\ge1} I_n^{(B)}(w)\,u^{2n-1}
&
\displaystyle \sum_{n\ge1}(2^{2n+1}-1)I_n^{(A)}(w)\,u^{2n}
\\ \hline
1
&
\displaystyle \frac{u}{4}\sec(u/2)
&
\displaystyle u\tan u
\\ \hline
x^2
&
\displaystyle \frac12\csc(u/2)-\frac{u}{4}\csc(u/2)\cot(u/2)
&
\displaystyle 1-u\cot u
\\ \hline
\sqrt{1-x^2}
&
\displaystyle \frac{\pi}{2}\csc u-\frac{\pi}{4}\csc(u/2)
&
\displaystyle \frac{\pi}{2}(\sec u-1)
\\ \hline
\displaystyle \frac{1}{\sqrt{1-x^2}}
&
\displaystyle \frac{\pi}{4}\tan(u/2)
&
\displaystyle \frac{\pi}{2}(\sec u-\cos u)
\\ \hline
x^2\sqrt{1-x^2}
&
\displaystyle \frac{\pi}{16}\tan(u/4)\sec^2(u/4)
&
\displaystyle \frac{\pi}{4}\tan^2(u/2)
\\ \hline
\displaystyle \frac{x^2}{\sqrt{1-x^2}}
&
\displaystyle \frac{\pi}{4}\tan(u/4)
&
\displaystyle \frac{\pi}{2}(1-\cos u)
\\ \hline
\end{array}
\]

The only non-immediate cases are the left-hand generating function for \(w(x)=x^2\), and both generating functions for \(w(x)=x^2\sqrt{1-x^2}\).

\vspace{0.2cm}

For \(w(x)=x^2\), we rewrite
\[
\frac12\csc(u/2)-\frac{u}{4}\csc(u/2)\cot(u/2)
=
\frac12\frac{d}{dz}\bigl(z\csc z\bigr)\Big|_{z=u/2}.
\]
Since
\[
\csc z
=
\sum_{n\ge0}\frac{(-1)^{n+1}2(2^{2n-1}-1)B_{2n}}{(2n)!}\,z^{2n-1},
\]
it follows that
\[
\frac12\frac{d}{dz}\bigl(z\csc z\bigr)
=
\sum_{n\ge1}
\frac{(-1)^{n+1}(2^{2n-1}-1)B_{2n}}{(2n-1)!}\,z^{2n-1},
\]
and substituting \(z=u/2\) gives the required series.

\vspace{0.2cm}

For \(w(x)=x^2\sqrt{1-x^2}\), we first of all read from \eqref{eq:series_of_tan_square} this series
\[
\tan^2 z
=
\sum_{n\ge1}
\frac{2^{2n+1}(2^{2n+2}-1)(-1)^nB_{2n+2}}{(n+1)(2n)!}\,z^{2n}.
\]

Furthermore, we use
\[
\frac{\pi}{16}\tan(u/4)\sec^2(u/4)
=
\frac{\pi}{8}\frac{d}{du}\!\left(\tan^2\frac{u}{4}\right).
\] and the series of $\tan^2 z$.

Likewise, the right-hand generating function
\[
\frac{\pi}{4}\tan^2(u/2)
\]
is obtained directly from the same expansion of \(\tan^2 z\) upon substituting \(z=u/2\).

All remaining coefficient extractions follow directly from the classical Taylor expansions of
\[
\sec z,\qquad \tan z,\qquad \csc z,\qquad z\cot z.
\]
\end{proof}

\part{Asymptotic zero distribution of the polynomials $\widetilde{\Xi}_n$ and $\widetilde{\Lambda}_n$}\label[part]{part:asymptotic_zero_distribution}

\section{Sums of roots of the polynomials}

\begin{proposition}\label{prop:sum-of-roots-tilde}
Let
\[
\widetilde{\Xi}_n(y):=\Xi_n(\sqrt y),
\qquad
\widetilde{\Lambda}_n(y):=\Lambda_n(\sqrt y).
\]
Then the sums of the zeros of \(\widetilde{\Xi}_n\) and \(\widetilde{\Lambda}_n\) are given by
\[
\sum_{\widetilde{\Xi}_n(\rho)=0}\rho=\frac{4n-3}{6},
\qquad
\sum_{\widetilde{\Lambda}_n(\rho)=0}\rho=\frac{2(n-1)}{3}.
\]
\end{proposition}

\begin{proof}
By \cref{theorem:explicit_polynomials}, we have
\[
\widetilde{\Xi}_n(y)
=
\frac{(-1)^{n+1}}{2^{4n-2}(2n-1)!}
\sum_{t=0}^{n-1}\mathcal C^{(B)}_{n,t}\,y^t,
\qquad
\widetilde{\Lambda}_n(y)
=
\frac{2(-1)^{n+1}}{(2^{2n+1}-1)(2n)!}
\sum_{t=0}^{n-1}\mathcal C^{(A)}_{n,t}\,y^t.
\]
Hence, by Viète's formula,
\[
\sum_{\widetilde{\Xi}_n(\rho)=0}\rho
=
-\frac{\mathcal C^{(B)}_{n,n-2}}{\mathcal C^{(B)}_{n,n-1}},
\qquad
\sum_{\widetilde{\Lambda}_n(\rho)=0}\rho
=
-\frac{\mathcal C^{(A)}_{n,n-2}}{\mathcal C^{(A)}_{n,n-1}}.
\]

Now \cref{prop:Cn-n-2} gives $\displaystyle \mathcal C^{(B)}_{n,n-2}=-\frac{4n-3}{3}\,2^{2n-3}(2n-1)!$ and $\displaystyle \mathcal C^{(A)}_{n,n-2}=-\frac{n-1}{3}(2n)!$. Moreover, by \cref{prop:leading_coefficients}, we have $\displaystyle [x^{2n-2}]\Xi_n(x)=\frac{(-1)^{n+1}}{2^{2n}}$ and $\displaystyle [x^{2n-2}]\Lambda_n(x)=\frac{(-1)^{n+1}}{2^{2n+1}-1}$.

Comparing these leading coefficients with the formulae above from \cref{theorem:explicit_polynomials}, we obtain $\displaystyle \mathcal C^{(B)}_{n,n-1}=2^{2n-2}(2n-1)!$ and $\displaystyle \mathcal C^{(A)}_{n,n-1}=\frac{(2n)!}{2}$. Therefore
\[
\sum_{\widetilde{\Xi}_n(\rho)=0}\rho
=
-\frac{-\frac{4n-3}{3}\,2^{2n-3}(2n-1)!}{2^{2n-2}(2n-1)!}
=
\frac{4n-3}{6},
\]
\[
\sum_{\widetilde{\Lambda}_n(\rho)=0}\rho
=
-\frac{-\frac{n-1}{3}(2n)!}{(2n)!/2}
=
\frac{2(n-1)}{3}.
\]
This proves the claim.
\end{proof}

\begin{remark}
It is worth noting that, upon replacing \(n\) by a real variable \(x\), the two
root-sum formulae become the affine functions
\(\displaystyle
\frac{4x-3}{6}
\qquad\text{and}\qquad
\frac{2(x-1)}{3},
\)
which are parallel, since both have slope \(2/3\); indeed, their difference is
identically \(1/6\).
\end{remark}

\section{Log derivativatives of the polynomials}

\begin{lemma}\label[lemma]{lem:log-derivative-Xi-Lambda}
Let $n\in\mathbb{N}$ and $0<x<1$. Set
\[
y(x):=-\,\frac{1-\sqrt{x}}{1+\sqrt{x}}.
\]
Then, for every $x$ with $B_{2n-1}(y(x))\neq0$,
\[
\frac{\widetilde{\Xi}_n'(x)}{\widetilde{\Xi}_n(x)}
=
\frac{2n-1}{2\sqrt{x}\,(1+\sqrt{x})}
-\frac{1}{2x}
-\frac{1}{4\sqrt{x}\,(1-\sqrt{x})}
\left[
  \frac{B_{2n}\!\bigl(y(x)\bigr)}{B_{2n-1}\!\bigl(y(x)\bigr)}
  -\bigl((4n-1)\,y(x)+1\bigr)
\right].
\]
Likewise, for every $x$ with $A_{2n}(y(x))\neq0$,
\[
\frac{\widetilde{\Lambda}_n'(x)}{\widetilde{\Lambda}_n(x)}
=
\frac{2n-1}{2\sqrt{x}\,(1+\sqrt{x})}
-\frac{1}{2x}
-\frac{1}{2\sqrt{x}\,(1-\sqrt{x})}
\left[
  \frac{A_{2n+1}\!\bigl(y(x)\bigr)}{A_{2n}\!\bigl(y(x)\bigr)}
  -\bigl(2n\,y(x)+1\bigr)
\right].
\]
\end{lemma}

\begin{proof}
By \cref{prop:eulerian_polynomial_form},
\[
\widetilde{\Xi}_n(x)
=
C_n\frac{(1+\sqrt{x})^{2n-1}}{\sqrt{x}}\,B_{2n-1}\!\bigl(y(x)\bigr),
\qquad
\widetilde{\Lambda}_n(x)
=
D_n\frac{(1+\sqrt{x})^{2n-1}}{\sqrt{x}}\,A_{2n}\!\bigl(y(x)\bigr),
\]
where
\[
C_n:=\frac{(-1)^{n+1}}{2^{4n-1}(2n-1)!},
\qquad
D_n:=\frac{(-1)^{n+1}}{(2^{2n+1}-1)(2n)!}.
\]
Also,
\[
y'(x)=\frac{1}{\sqrt{x}(1+\sqrt{x})^2}.
\]

Hence, for
\(\displaystyle
P_n(z):=B_{2n-1}(z)\quad\text{or}\quad P_n(z):=A_{2n}(z),
\)
we have
\[
\frac{d}{dx}\log\!\left(
\frac{(1+\sqrt{x})^{2n-1}}{\sqrt{x}}\,P_n(y(x))
\right)
=
\frac{2n-1}{2\sqrt{x}(1+\sqrt{x})}
-\frac{1}{2x}
+
\frac{1}{\sqrt{x}(1+\sqrt{x})^2}\,
\frac{P_n'(y(x))}{P_n(y(x))}.
\]

It remains to use the recursive formulae for the Eulerian polynomials.

For the type $B$ Eulerian polynomials, the recursive formula is
\[
B_m(z)=\bigl((2m-1)z+1\bigr)B_{m-1}(z)+2z(1-z)B_{m-1}'(z).
\]
After replacing $m$ by $m+1$ and solving for $B_m'(z)/B_m(z)$, we get
\[
\frac{B_m'(z)}{B_m(z)}
=
\frac{1}{2z(1-z)}
\left(
\frac{B_{m+1}(z)}{B_m(z)}-\bigl((2m+1)z+1\bigr)
\right).
\]
With $m=2n-1$ and $z=y(x)$,
\[
\frac{B_{2n-1}'(y(x))}{B_{2n-1}(y(x))}
=
\frac{1}{2y(x)(1-y(x))}
\left(
\frac{B_{2n}(y(x))}{B_{2n-1}(y(x))}
-\bigl((4n-1)y(x)+1\bigr)
\right).
\]

For the type $A$ Eulerian polynomials, the recursive formula is
\[
A_{m+1}(z)=(mz+1)A_m(z)+z(1-z)A_m'(z).
\]
Solving for $A_m'(z)/A_m(z)$ gives
\[
\frac{A_m'(z)}{A_m(z)}
=
\frac{1}{z(1-z)}
\left(
\frac{A_{m+1}(z)}{A_m(z)}-(mz+1)
\right).
\]
With $m=2n$ and $z=y(x)$,
\[
\frac{A_{2n}'(y(x))}{A_{2n}(y(x))}
=
\frac{1}{y(x)(1-y(x))}
\left(
\frac{A_{2n+1}(y(x))}{A_{2n}(y(x))}
-\bigl(2n\,y(x)+1\bigr)
\right).
\]

Finally,
\[
y(x)(1-y(x))=-\frac{2(1-\sqrt{x})}{(1+\sqrt{x})^2},
\]
so
\[
\frac{1}{\sqrt{x}(1+\sqrt{x})^2}\frac{1}{2y(x)(1-y(x))}
=
-\frac{1}{4\sqrt{x}(1-\sqrt{x})},
\qquad
\frac{1}{\sqrt{x}(1+\sqrt{x})^2}\frac{1}{y(x)(1-y(x))}
=
-\frac{1}{2\sqrt{x}(1-\sqrt{x})}.
\]
Substituting these identities into the logarithmic derivative formula yields the claim.
\end{proof}

\section{Limit of the Stieltjes transform, limit density and limit distribution}

For \(n\in\mathbb{N}\), let \(\mu_n^{\Xi}\) and \(\mu_n^{\Lambda}\) denote the empirical zero measures of \(\widetilde{\Xi}_n\) and \(\widetilde{\Lambda}_n\), respectively. Define
\[
s_n^{\Xi}(z):=\frac{1}{n-1}\frac{\widetilde{\Xi}_n'(z)}{\widetilde{\Xi}_n(z)},
\qquad
s_n^{\Lambda}(z):=\frac{1}{n-1}\frac{\widetilde{\Lambda}_n'(z)}{\widetilde{\Lambda}_n(z)},
\qquad z\in\mathbb{C}\setminus[0,1].
\]
Then \(\displaystyle s_n^{\Xi}\) and \(\displaystyle s_n^{\Lambda}\) are the Stieltjes transforms of \(\displaystyle \mu_n^{\Xi}\) and \(\displaystyle \mu_n^{\Lambda}\), respectively, that is,
\begin{equation}\label{eq:sn-Stieltjes-Xi-Lambda}
s_n^{\Xi}(z)=\int_{\mathbb{R}}\frac{1}{z-x}\,\mu_n^{\Xi}(dx),
\qquad
s_n^{\Lambda}(z)=\int_{\mathbb{R}}\frac{1}{z-x}\,\mu_n^{\Lambda}(dx),
\qquad z\in\mathbb{C}\setminus[0,1].
\end{equation}

\begin{lemma}\label[lemma]{lem:s-n-limit-Xi-Lambda}
Let
\[
\Omega:=\mathbb C\setminus(-\infty,1],
\]
and define
\[
s_n^{\Xi}(z):=\frac{1}{n-1}\frac{\widetilde{\Xi}_n'(z)}{\widetilde{\Xi}_n(z)},
\qquad
s_n^{\Lambda}(z):=\frac{1}{n-1}\frac{\widetilde{\Lambda}_n'(z)}{\widetilde{\Lambda}_n(z)},
\qquad z\in\Omega.
\]
Fix the branch of \(\sqrt z\) on \(\Omega\) satisfying \(\Re(\sqrt z)>0\), and set
\(
u(z):=\dfrac{\sqrt z-1}{\sqrt z+1}.
\)

Then \(|u(z)|<1\) for every \(z\in\Omega\), and
\(
s_n^{\Xi}(z)\longrightarrow s(z),
\qquad
s_n^{\Lambda}(z)\longrightarrow s(z),
\)
locally uniformly on \(\Omega\), where
\[
s(z)
=
\frac{2}{\sqrt{z}\,(1-z)\,\log u(z)}.
\]
In particular, the two limiting functions coincide.
\end{lemma}

\begin{proof}
By \cref{lem:log-derivative-Xi-Lambda}, for \(z\in\Omega\),
\[
\frac{\widetilde{\Xi}_n'(z)}{\widetilde{\Xi}_n(z)}
=
\frac{2n-1}{2\sqrt{z}\,(1+\sqrt{z})}
-\frac{1}{2z}
-\frac{1}{4\sqrt{z}\,(1-\sqrt{z})}
\left[
\frac{B_{2n}(u(z))}{B_{2n-1}(u(z))}
-\bigl((4n-1)u(z)+1\bigr)
\right],
\]
\[
\frac{\widetilde{\Lambda}_n'(z)}{\widetilde{\Lambda}_n(z)}
=
\frac{2n-1}{2\sqrt{z}\,(1+\sqrt{z})}
-\frac{1}{2z}
-\frac{1}{2\sqrt{z}\,(1-\sqrt{z})}
\left[
\frac{A_{2n+1}(u(z))}{A_{2n}(u(z))}
-\bigl(2n\,u(z)+1\bigr)
\right].
\]

Dividing by \(n-1\), we obtain
\[
\begin{aligned}
s_n^{\Xi}(z)
&=
\frac{2n-1}{n-1}\,\frac{1}{2\sqrt{z}\,(1+\sqrt{z})}
-\frac{1}{2z(n-1)} \\
&\quad
-\frac{1}{4\sqrt{z}\,(1-\sqrt{z})}
\left[
\frac{1}{n-1}\frac{B_{2n}(u(z))}{B_{2n-1}(u(z))}
-\frac{4n-1}{n-1}u(z)
-\frac{1}{n-1}
\right],
\end{aligned}
\]
\[
\begin{aligned}
s_n^{\Lambda}(z)
&=
\frac{2n-1}{n-1}\,\frac{1}{2\sqrt{z}\,(1+\sqrt{z})}
-\frac{1}{2z(n-1)} \\
&\quad
-\frac{1}{2\sqrt{z}\,(1-\sqrt{z})}
\left[
\frac{1}{n-1}\frac{A_{2n+1}(u(z))}{A_{2n}(u(z))}
-\frac{2n}{n-1}u(z)
-\frac{1}{n-1}
\right].
\end{aligned}
\]

Since \(\Re(\sqrt z)>0\) on \(\Omega\), we have
\[
|u(z)|=\left|\frac{\sqrt z-1}{\sqrt z+1}\right|<1
\qquad (z\in\Omega).
\]
Therefore \cref{cor:A-B-ratio} applies with \(x=u(z)\), and yields
\[
\frac{1}{2n-1}\frac{B_{2n}(u(z))}{B_{2n-1}(u(z))}
\longrightarrow
\frac{2(1-u(z))}{-\log u(z)},
\]
\[
\frac{1}{2n}\frac{A_{2n+1}(u(z))}{A_{2n}(u(z))}
\longrightarrow
\frac{1-u(z)}{-\log u(z)},
\]
locally uniformly on \(\Omega\). Hence
\[
\frac{1}{n-1}\frac{B_{2n}(u(z))}{B_{2n-1}(u(z))}
\longrightarrow
\frac{4(1-u(z))}{-\log u(z)},
\qquad
\frac{1}{n-1}\frac{A_{2n+1}(u(z))}{A_{2n}(u(z))}
\longrightarrow
\frac{2(1-u(z))}{-\log u(z)},
\]
locally uniformly on \(\Omega\).

Moreover,
\[
\frac{2n-1}{n-1}\to 2,
\qquad
\frac{4n-1}{n-1}\to 4,
\qquad
\frac{2n}{n-1}\to 2,
\qquad
\frac{1}{n-1}\to 0.
\]
Substituting these limits, we obtain
\[
s_n^{\Xi}(z)\longrightarrow
\frac{1}{\sqrt{z}\,(1+\sqrt{z})}
-\frac{1}{4\sqrt{z}\,(1-\sqrt{z})}
\left[
\frac{4(1-u(z))}{-\log u(z)}-4u(z)
\right],
\]
\[
s_n^{\Lambda}(z)\longrightarrow
\frac{1}{\sqrt{z}\,(1+\sqrt{z})}
-\frac{1}{2\sqrt{z}\,(1-\sqrt{z})}
\left[
\frac{2(1-u(z))}{-\log u(z)}-2u(z)
\right],
\]
locally uniformly on \(\Omega\). Both limits simplify to
\[
s(z)
=
\frac{1}{\sqrt{z}\,(1+\sqrt{z})}
+
\frac{1}{\sqrt{z}\,(1-\sqrt{z})}
\left(
u(z)+\frac{1-u(z)}{\log u(z)}
\right) = \frac{2}{\sqrt{z}\,(1-z)\,\log u(z)}.
\]
This proves the claim.
\end{proof}

\begin{theorem}[Limiting density]\label{thm:limiting-density}
Let \(\mu_n^{\Xi}\) and \(\mu_n^{\Lambda}\) be the empirical zero measures associated with
\(\widetilde{\Xi}_n\) and \(\widetilde{\Lambda}_n\), respectively. Then both sequences
\(\mu_n^{\Xi}\) and \(\mu_n^{\Lambda}\) converge weakly, as \(n\to\infty\), to the same
probability measure \(\mu\), absolutely continuous on \((0,1)\), with density
\begin{equation}\label{eq:rho-density}
\rho(x)
=
\frac{2}{\sqrt{x}(1-x)\left(\log^2\!\frac{1-\sqrt{x}}{1+\sqrt{x}}+\pi^2\right)},
\qquad 0<x<1,
\end{equation}
and \(\rho(x)=0\) for \(x\notin(0,1)\).
\end{theorem}

\begin{proof}
Define
\[
s_n^{\Xi}(z):=\frac{1}{n-1}\frac{\widetilde{\Xi}_n'(z)}{\widetilde{\Xi}_n(z)},
\qquad
s_n^{\Lambda}(z):=\frac{1}{n-1}\frac{\widetilde{\Lambda}_n'(z)}{\widetilde{\Lambda}_n(z)}.
\]
Since all zeros of \(\widetilde{\Xi}_n\) and \(\widetilde{\Lambda}_n\) lie in \((0,1)\), the functions
\(s_n^{\Xi}\) and \(s_n^{\Lambda}\) are the Stieltjes transforms of \(\mu_n^{\Xi}\) and
\(\mu_n^{\Lambda}\), respectively, and are analytic on \(\mathbb{C}\setminus[0,1]\).

Let \(\Omega:=\mathbb{C}\setminus(-\infty,1]\). By \cref{lem:s-n-limit-Xi-Lambda},
\[
s_n^{\Xi}(z)\to s(z),
\qquad
s_n^{\Lambda}(z)\to s(z),
\]
locally uniformly on \(\Omega\), where
\begin{equation}\label{eq:s-limit}
s(z)
=
\frac{2}{\sqrt{z}\,(1-z)\,\log u(z)},
\end{equation}
with
\(
u(z):=\dfrac{\sqrt{z}-1}{\sqrt{z}+1},
\)
and where \(\sqrt{z}\) denotes the branch on \(\Omega\) satisfying \(\Re(\sqrt z)>0\).
In particular, \(s\) is analytic on \(\Omega\), hence on the upper half-plane.

Since \(\mu_n^{\Xi}\) and \(\mu_n^{\Lambda}\) are probability measures, their Stieltjes transforms satisfy
\[
s_n^{\Xi}(z)=\frac{1}{z}+O(z^{-2}),
\qquad
s_n^{\Lambda}(z)=\frac{1}{z}+O(z^{-2}),
\qquad z\to\infty.
\]
Passing to the limit, we obtain
\[
s(z)=\frac{1}{z}+O(z^{-2}),
\qquad z\to\infty.
\]
Therefore \(s\) is the Stieltjes transform of a unique probability measure \(\mu\) supported on \([0,1]\), and
\[
\mu_n^{\Xi}\Rightarrow\mu,
\qquad
\mu_n^{\Lambda}\Rightarrow\mu
\]
weakly.

It remains to identify the density of \(\mu\) on \((0,1)\). Fix \(x\in(0,1)\) and set
\(
q(x):=\dfrac{1-\sqrt{x}}{1+\sqrt{x}}\in(0,1).
\)
As \(\varepsilon\to0^+\), we have \(u(x+i\varepsilon)\to -q(x)\), and with the branch of the logarithm induced from \(\Omega\),
\[
\lim_{\varepsilon\to0^+}\log u(x+i\varepsilon)=\log q(x)+i\pi.
\]
Hence
\[
\lim_{\varepsilon\to0^+}\frac{1}{\log u(x+i\varepsilon)}
=
\frac{\log q(x)-i\pi}{\log^2 q(x)+\pi^2}.
\]

Moreover,
\[
\lim_{\varepsilon\to0^+}(1-u(x+i\varepsilon))
=
1+q(x)
=
\frac{2}{1+\sqrt{x}},
\]
so that
\[
\lim_{\varepsilon\to0^+}\frac{1-u(x+i\varepsilon)}{\sqrt{x}(1-\sqrt{x})}
=
\frac{2}{\sqrt{x}(1-x)}.
\]
All remaining terms in \eqref{eq:s-limit} have real limits as \(\varepsilon\to0^+\). Therefore
\[
\Im\!\left(\lim_{\varepsilon\to0^+} s(x+i\varepsilon)\right)
=
\frac{2}{\sqrt{x}(1-x)}
\Im\!\left(\frac{1}{\log q(x)+i\pi}\right)
=
-\frac{2\pi}{\sqrt{x}(1-x)\bigl(\log^2 q(x)+\pi^2\bigr)}.
\]
By the Stieltjes inversion formula,
\(
\rho(x)
=
-\dfrac{1}{\pi}\Im\!\left(\displaystyle\lim_{\varepsilon\to0^+} s(x+i\varepsilon)\right),
\qquad 0<x<1.
\)
Thus
\[
\rho(x)
=
\frac{2}{\sqrt{x}(1-x)\bigl(\log^2 q(x)+\pi^2\bigr)}.
\]
Since \(q(x)=\dfrac{1-\sqrt{x}}{1+\sqrt{x}}\), this is exactly \eqref{eq:rho-density}. This proves the claim.
\end{proof}

\begin{theorem}[Limiting zero distribution]\label{thm:limiting-zero-distribution}
Let \(\mu_n^{\Xi}\) and \(\mu_n^{\Lambda}\) be the empirical zero measures associated with
\(\widetilde{\Xi}_n\) and \(\widetilde{\Lambda}_n\), respectively. Let
\(\displaystyle F_n^{\Xi}\) and \(\displaystyle F_n^{\Lambda}\) denote their distribution functions, i.e.
\[
F_n^{\Xi}(x):=\mu_n^{\Xi}((-\infty,x]),
\qquad
F_n^{\Lambda}(x):=\mu_n^{\Lambda}((-\infty,x]),
\qquad x\in\mathbb{R}.
\]
Then \(F_n^{\Xi}(x)\) and \(F_n^{\Lambda}(x)\) converge pointwise on \(\mathbb{R}\) to the same function
\[
F(x):=
\begin{cases}
0, & x\le 0,\\[1mm]
\displaystyle \frac{2}{\pi}\arctan\!\left(
  \frac{1}{\pi}\log\frac{1+\sqrt{x}}{1-\sqrt{x}}
\right), & 0<x<1,\\[3mm]
1, & x\ge 1.
\end{cases}
\]
In particular, the empirical distribution functions associated with the real zeros of
\(\widetilde{\Xi}_n\) and \(\widetilde{\Lambda}_n\) in \((0,1)\) converge to \(F\) at every
continuity point of \(F\).
\end{theorem}

\begin{proof}
By \cref{thm:limiting-density}, the measures \(\mu_n^{\Xi}\) and \(\mu_n^{\Lambda}\) converge
weakly to the same probability measure \(\mu\) supported on \((0,1)\) with density
\[
\rho(x)=\frac{2}{
  \sqrt{x}\,(1-x)
  \bigl(\log^2\!\tfrac{1-\sqrt{x}}{1+\sqrt{x}}+\pi^2\bigr)
},
\qquad 0<x<1,
\]
and \(\rho(x)=0\) for \(x\notin(0,1)\). Hence the distribution function of \(\mu\) is
\[
F(x)=\mu((-\infty,x])=
\begin{cases}
0, & x\le 0,\\[1mm]
\displaystyle \int_0^x \rho(t)\,dt, & 0<x<1,\\[3mm]
1, & x\ge 1.
\end{cases}
\]

For \(0<x<1\), define
\[
G(x):=\frac{2}{\pi}\arctan\!\left(
\frac{1}{\pi}\log\frac{1+\sqrt{x}}{1-\sqrt{x}}
\right).
\]
Then \(G(0+)=0\), and a direct computation shows that \(G'(x)=\rho(x)\) on \((0,1)\).
Therefore,
\[
G(x)=\int_0^x\rho(t)\,dt,
\qquad 0<x<1,
\]
so that
\[
F(x)=
\begin{cases}
0, & x\le 0,\\[1mm]
\displaystyle \frac{2}{\pi}\arctan\!\left(
  \frac{1}{\pi}\log\frac{1+\sqrt{x}}{1-\sqrt{x}}
\right), & 0<x<1,\\[3mm]
1, & x\ge 1.
\end{cases}
\]

Since \(\mu_n^{\Xi}\Rightarrow\mu\) and \(\mu_n^{\Lambda}\Rightarrow\mu\), and since \(F\) is
continuous on \(\mathbb{R}\), standard results on weak convergence of probability measures
(e.g. \cite[Thm.~2.3]{Billingsley-convergence}) imply that
\[
F_n^{\Xi}(x)\to F(x),
\qquad
F_n^{\Lambda}(x)\to F(x),
\qquad x\in\mathbb{R}.
\]
This proves the claim.
\end{proof}

\section{Interpretation}

Using \cref{thm:limiting-density,thm:limiting-zero-distribution,cor:exact_degree}, one obtains a rather precise asymptotic description of the real zeros of both families \(\widetilde{\Xi}_n\) and \(\widetilde{\Lambda}_n\). Indeed, \cref{cor:exact_degree} shows that each of these polynomials has exactly \(n-1\) real zeros. Writing them in increasing order as
\[
0<x_{1,n}<x_{2,n}<\cdots<x_{n-1,n}<1,
\]
the weak convergence of the corresponding empirical zero measures implies that their normalized counting measures are governed by the same limiting probability measure \(\mu\) on \((0,1)\) with density~\(\rho\). Equivalently, the associated distribution functions converge to the same limiting distribution function \(F\). Hence the ordered zeros satisfy the asymptotic quantile relation
\[
F(x_{k,n})\approx \frac{k}{\,n-1\,},
\qquad n\to\infty,
\]
and therefore admit the asymptotic localization
\[
x_{k,n}\approx F^{-1}\!\left(\frac{k}{\,n-1\,}\right),
\qquad 1\le k\le n-1.
\]
Since \(F\) is explicit, this can be written more concretely as
\(\displaystyle
x_{k,n}\approx
\tanh^2\!\left(
\frac{\pi}{2}\tan\!\left(\frac{\pi k}{2(n-1)}\right)
\right).
\)
In this sense, the function \(F\), or equivalently its inverse \(F^{-1}\), provides the asymptotic profile of the full ordered zero set.

\vspace{0.2cm}

The density \(\rho\) shows that zeros accumulate strongly near both endpoints of the interval. As \(x\downarrow0\),
\[
\rho(x)\sim \frac{2}{\pi^2\sqrt{x}},
\]
so the density has an integrable square-root singularity at the origin. Integrating this asymptotic form gives
\[
F(x)\approx \frac{4}{\pi^2}\sqrt{x},
\qquad x\downarrow0,
\]
and inverting this relation yields, for fixed \(k\) and \(n\to\infty\),
\(\displaystyle
x_{k,n}\sim \frac{\pi^4}{16}\,\frac{k^2}{(n-1)^2}.
\)
Thus the smallest zeros are located on the \(n^{-2}\)-scale. More generally, the zeros near the left edge exhibit a quadratic, or parabolic, spacing pattern.

\vspace{0.2cm}

The situation near \(x=1\) is markedly different. Writing
\(\displaystyle
L(x):=\log\frac{1+\sqrt{x}}{1-\sqrt{x}},
\)
we have \(L(x)\to+\infty\) as \(x\uparrow1\), and the explicit formula for \(F\) implies
\[
1-F(x)\sim \frac{2}{L(x)},
\qquad x\uparrow1.
\]
Since \(L(x)\) grows only logarithmically in the distance to \(1\), the limiting distribution has a very slowly decaying right tail. If \(n-1-k\) remains fixed while \(n\to\infty\), then heuristically
\[
1-F(x_{k,n})\approx \frac{n-1-k}{\,n-1\,}\sim \frac{2}{L(x_{k,n})},
\]
which shows that \(L(x_{k,n})\) is of order \(n\). Consequently, \(1-x_{k,n}\) is exponentially small in \(n\). In other words, the largest zeros approach \(1\) at an essentially exponential rate, much faster than any algebraic scale.

\vspace{0.2cm}

To summarize, the real zeros of \(\widetilde{\Xi}_n\) and \(\widetilde{\Lambda}_n\) both concentrate on \((0,1)\) according to the same limiting law. The density becomes singular at both endpoints, but the two edges have different character: near \(0\) one sees a square-root singularity, while near \(1\) the accumulation is governed by a logarithmic effect. Accordingly, the smallest zeros are of order \(n^{-2}\), whereas the largest zeros lie exponentially close to \(1\). Away from the endpoints, the zeros are distributed on the usual \(1/n\)-scale and are described by the smooth bulk profile determined by~\(\rho\).

\section{Numerical experiments}

We conclude with a numerical illustration of the cumulative distribution functions corresponding to the real zeros of \(\widetilde{\Xi}_n\) and \(\widetilde{\Lambda}_n\) in \((0,1)\).

Our emphasis is not on computing individual zeros to very high accuracy, but rather on comparing the empirical distribution functions \(F_n\) with the limiting distribution function \(F\). The resulting plots provide a clear visual confirmation of the asymptotic results established in \cref{thm:limiting-density,thm:limiting-zero-distribution}: as \(n\) grows, the empirical CDFs approach the same limiting curve more and more closely.

One also sees the characteristic boundary behavior predicted by the limiting law. Near \(0\), the smallest zeros move toward the origin on the natural \(n^{-2}\)-scale, whereas near \(1\), the largest zeros accumulate very quickly, producing a sharp increase of the CDF close to the right endpoint.

\begin{center}
	\includegraphics[width=0.95\textwidth]{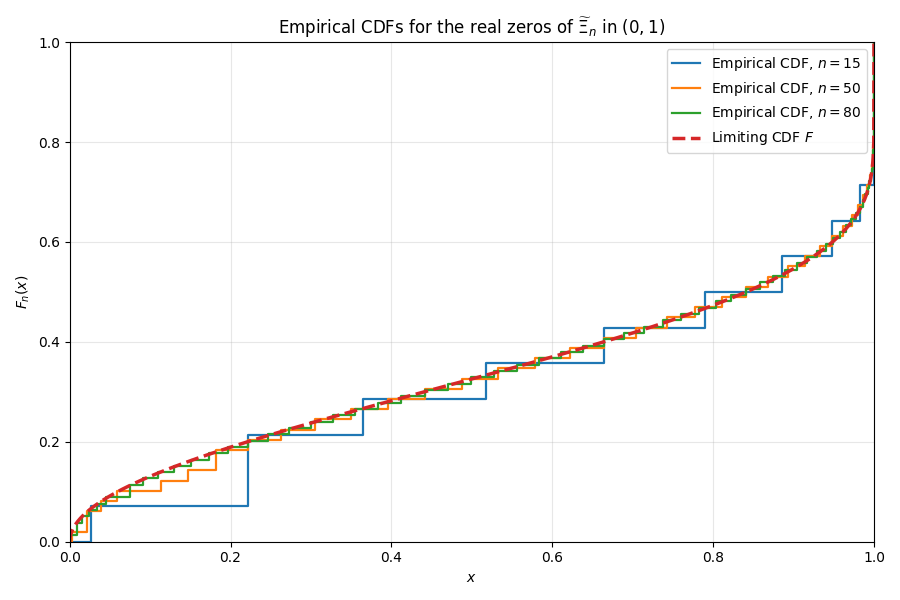}
\end{center}

\begin{center}
	\includegraphics[width=0.95\textwidth]{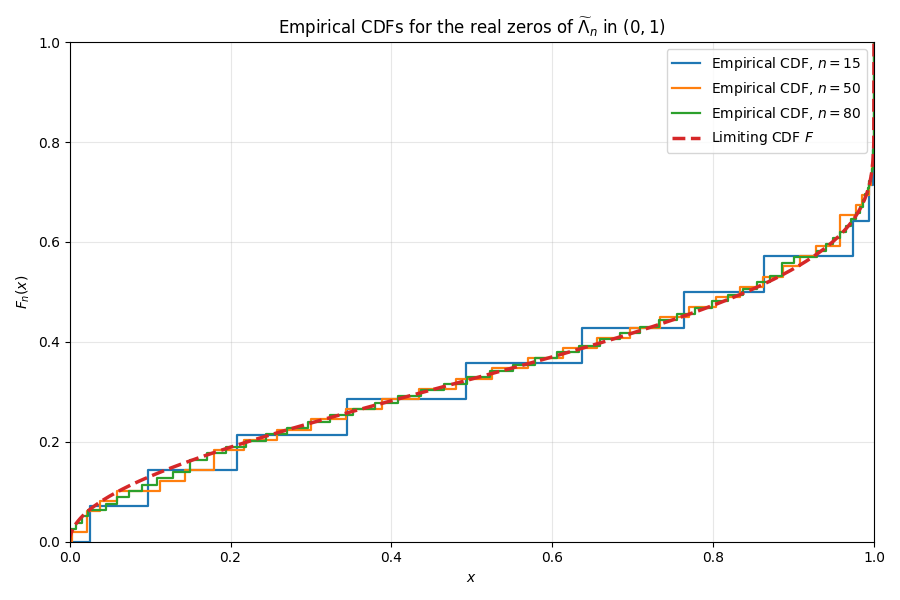}
\end{center}

\begin{proposition}\label{prop:log-tanh-integral}
One has
\[
\int_{0}^{\infty}\frac{\log \tanh\!\left(\frac{\pi x}{2}\right)}{1+x^{2}}\,dx
=
\frac{\pi}{4}\log\!\frac{4}{\pi^{2}}
=
-\frac{\pi}{2}\log\!\frac{\pi}{2}.
\]
\end{proposition}

\begin{proof}
Let $\displaystyle J:=\int_{0}^{\infty}\frac{\log \tanh\!\left(\frac{\pi x}{2}\right)}{1+x^{2}}\,dx$.
If $\displaystyle \rho_{1,n},\dots,\rho_{n-1,n}\in(0,1)$ are the zeros of $\widetilde{\Xi}_n$ and $\displaystyle \mu_n^{\Xi}:=\frac1{n-1}\sum_{k=1}^{n-1}\delta_{\rho_{k,n}}$, then Viète's formula gives
$\displaystyle \prod_{k=1}^{n-1}\rho_{k,n}=\left|\frac{\widetilde{\Xi}_n(0)}{\operatorname{lc}(\widetilde{\Xi}_n)}\right|$,
hence
$\displaystyle \int_0^1 \log t\,\mu_n^{\Xi}(dt)=\frac1{n-1}\sum_{k=1}^{n-1}\log\rho_{k,n}
=\frac1{n-1}\log\left|\frac{\widetilde{\Xi}_n(0)}{\operatorname{lc}(\widetilde{\Xi}_n)}\right|$.
By \cref{prop:endpoint-ratio-limit},
$\displaystyle \lim_{n\to\infty}\left|\frac{\widetilde{\Xi}_n(0)}{\operatorname{lc}(\widetilde{\Xi}_n)}\right|^{1/(n-1)}=\frac{4}{\pi^2}$,
so
$\displaystyle \lim_{n\to\infty}\int_0^1 \log t\,\mu_n^{\Xi}(dt)=\log\frac{4}{\pi^2}$.

By \cref{thm:limiting-density}, $\mu_n^{\Xi}\Rightarrow\mu$, where $\mu$ has density
$\displaystyle \rho(t)=\frac{2}{\sqrt{t}(1-t)\left(\log^2\!\frac{1-\sqrt{t}}{1+\sqrt{t}}+\pi^2\right)}$ on $0<t<1$.
Thus $\displaystyle \int_0^1 \log t\,\rho(t)\,dt=\log\frac{4}{\pi^2}$.
Now set $\displaystyle u=\log\frac{1+\sqrt{t}}{1-\sqrt{t}}$, so that $\displaystyle t=\tanh^2\!\left(\frac{u}{2}\right)$, $\displaystyle dt=\sqrt{t}(1-t)\,du$, $\displaystyle \log t=2\log\tanh\!\left(\frac{u}{2}\right)$, and $\displaystyle \log\frac{1-\sqrt{t}}{1+\sqrt{t}}=-u$.
Therefore $\displaystyle \rho(t)\,dt=\frac{2\,du}{u^2+\pi^2}$, and hence
\[
\int_0^1 \log t\,\rho(t)\,dt
=
4\int_0^\infty \frac{\log\tanh(u/2)}{u^2+\pi^2}\,du
=
\frac{4}{\pi}\int_0^\infty \frac{\log\tanh(\pi x/2)}{1+x^2}\,dx
=
\frac{4}{\pi}J.
\]
Combining this with $\displaystyle \int_0^1 \log t\,\rho(t)\,dt=\log\frac{4}{\pi^2}$ gives
$\displaystyle \frac{4}{\pi}J=\log\frac{4}{\pi^2}$, hence
$\displaystyle J=\frac{\pi}{4}\log\frac{4}{\pi^2}=-\frac{\pi}{2}\log\frac{\pi}{2}$.
\end{proof}

\begin{remark}
The evaluation
\[
\int_{0}^{\infty}\frac{\log \tanh\!\left(\frac{\pi x}{2}\right)}{1+x^{2}}\,dx
=
-\frac{\pi}{2}\log\!\frac{\pi}{2}
\]
is a special case of a more general logarithmic hyperbolic tangent integral formula available in the literature; see, for example, Reynolds and Stauffer~ \textnormal{\cite{ReynoldsStauffer2020}}. More precisely, in their general identity
\[
\int_0^\infty \frac{\log\!\bigl(\tanh(\alpha y/2)\bigr)}{b^2+y^2}\,dy
=
\frac{\pi}{2b}\log\!\left[
\frac{b\alpha}{2\pi}
\left(
\frac{\Gamma\!\left(\frac{\pi+b\alpha}{2\pi}\right)}
{\Gamma\!\left(1+\frac{b\alpha}{2\pi}\right)}
\right)^2
\right],
\]
one takes
\(
b=1,
\qquad
\alpha=\pi.
\) This gives
\[
\int_0^\infty \frac{\log \tanh(\pi y/2)}{1+y^2}\,dy
=
\frac{\pi}{2}\log\!\left[
\frac12
\left(\frac{\Gamma(1)}{\Gamma(3/2)}\right)^2
\right]
=
\frac{\pi}{2}\log\!\frac{2}{\pi}
=
-\frac{\pi}{2}\log\!\frac{\pi}{2}.
\]
The point of \cref{prop:log-tanh-integral} is therefore not the novelty of the integral evaluation itself, but rather that the same value follows naturally from \cref{prop:endpoint-ratio-limit,thm:limiting-density} and the asymptotic zero distribution of the polynomials \(\widetilde{\Xi}_n\).
\end{remark}

\section*{Acknowledgments}
The author acknowledges the use of an AI language model for assistance with literature search, presentation of the manuscript, verification of results and clarifying standard combinatorial identities.

\printbibliography

\end{document}